\documentclass[15pt]{article}

\usepackage{amsmath,amsfonts,amssymb,amsthm,amscd,mathrsfs}
\usepackage{enumitem}

\usepackage{yhmath}
\usepackage{graphicx} 
\newcommand\bigsubset[1][1.19]{%
   \mathrel{\vcenter{\hbox{\scalebox{#1}{$\subset$}}}}}

\newcounter{stepctr}
{\end{list}}

\newtheorem{thm}{Theorem}[section]
\newtheorem{prop}[thm]{Proposition}
\newtheorem{cor}[thm]{Corollary}

\theoremstyle{definition}
\newtheorem{dfn}[thm]{Definition}
\newtheorem{ex}[thm]{Example}

\newtheorem{rema}[thm]{Remark}
\newtheorem{lem}[thm]{Lemma}
\newtheorem{prob*}{Open problem}
\newcommand{\demo}{\begin{proof}}

\newcommand{\R}{\ensuremath{\mathcal{R}}}

\newcommand{\N}{\mathbb{N}}

\newcommand{\C}{\mathbb{C}}

\def\ll^2{{\mathcal L}(\ell^2(\N))}

\def\f^0x{{\mathcal F^0}(X) }

\title
{\bf  On the $g_{z}$-Kato decomposition and    generalization of  Koliha  Drazin invertibility}

\author{Z. Aznay, A. Ouahab, H. Zariouh}

\date{}
\begin{document}

\maketitle \thispagestyle{empty}

\begin{abstract}\noindent\baselineskip=10pt
In  \cite{koliha}, Koliha  proved  that   $T\in L(X)$ ($X$ is a complex  Banach space) is   generalized Drazin invertible operator       equivalent to  there exists   an      operator $S$ commuting  with $T$   such that $STS = S$ and    $\sigma(T^{2}S - T)\subset\{0\}$ which is  equivalent to say that $0\not\in \mbox{acc}\,\sigma(T).$  Later, in \cite{rwassa,rwassa1}   the authors     extended the class of generalized  Drazin invertible operators and they also  extended    the class of pseudo-Fredholm operators introduced by Mbekhta \cite{mbekhta} and other classes  of   semi-Fredholm   operators.
 As a continuation of these works, we introduce and study   the class of $g_{z}$-invertible (resp.,  $g_{z}$-Kato) operators which generalizes the class of generalized Drazin invertible operators (resp.,   the class  of    generalized Kato-meromorphic operators   introduced by  \v{Z}ivkovi\'{c}-Zlatanovi\'{c} and  Duggal in \cite{rwassa2}). Among other results, we prove  that  $T$ is  $g_{z}$-invertible   if and only if   $T$ is $g_{z}$-Kato with $\tilde{p}(T)=\tilde{q}(T)<\infty$  which is equivalent to   there exists    an      operator $S$ commuting  with $T$     such that $STS = S$ and     $\mbox{acc}\,\sigma(T^{2}S - T)\subset\{0\}$ which   in turn    is equivalent  to say that  $0\not\in \mbox{acc}\,(\mbox{acc}\,\sigma(T)).$     As application and  using  the concept  of the Weak SVEP introduced   at the end  of  this paper, we  give new  characterizations  of Browder-type theorems.

 \end{abstract}

 \baselineskip=15pt
  \baselineskip=15pt

\section{Introduction}
 Let $T \in L(X);$ where   $L(X)$ is the  Banach algebra of bounded linear operators acting on an infinite dimensional complex Banach space $(X,\|.\|).$  Throughout this paper   $T^{*},$ $\alpha(T)$ and  $\beta(T)$ means   respectively,     the dual of $T,$   the dimension of the kernel $\mathcal{N}(T)$  and the codimension of the range $\R(T).$ The ascent and the descent of $T$ are defined  by $p(T):=\inf\{n\in \mathbb{N}: \mathcal{N}(T^n) = \mathcal{N}(T^{n+1})\}$ (with $\mbox{inf}\emptyset=\infty$)    and  $q(T):= \inf\{n\in \mathbb{N}: \R(T^n) = \R(T^{n+1})\}.$  A  subspace $M$ of $X$ is $T$-invariant if $T(M)\subset M$ and   the restriction of $T$ on $M$ is denoted by $T_{M}.$   $(M,N) \in \mbox{Red}(T)$  if $M,$   $N$ are  closed $T$-invariant subspaces and  $X=M\oplus N$      ($M\oplus N$ means that $M\cap N=\{0\}$).    Let  $n\in\N,$   denote by   $T_{[n]}:=T_{\mathcal{R}(T^{n})}$    and by     $m_T:=\mbox{inf}\{n \in \N  :  \mbox{inf}\{\alpha(T_{[n]}),\beta(T_{[n]})\}<\infty\}$   the \textit{essential  degree} of   $T.$  According to \cite{berkani-sarih,mbekhta-muller}, $T$ is called upper semi-B-Fredholm  (resp., lower semi-B-Fredholm)  if the \textit{essential ascent} $p_{e}(T):=\mbox{inf}\{n \in \N  :  \alpha(T_{[n]})<\infty\}<\infty$ and $\R(T^{p_{e}(T)+1})$ is closed  (resp.,  the \textit{essential descent}  $q_{e}(T):=\mbox{inf}\{n \in \N  :  \beta(T_{[n]})<\infty\}<\infty$ and $\R(T^{q_{e}(T)})$ is closed). If  $T$ is   an upper or a lower  (resp.,  upper and   lower)  semi-B-Fredholm then $T$ it is called   \emph{semi-B-Fredholm} (resp., \emph{B-Fredholm})  and      its     index   is  defined   by $\mbox{ind}(T) := \alpha(T_{[m_{T}]})-\beta(T_{[m_{T}]}).$ $T$ is said to be an upper semi-B-Weyl (resp., a lower semi-B-Weyl, B-Weyl,  left Drazin invertible, right Drazin invertible, Drazin invertible) if  $T$ is an upper semi-B-Fredholm with $\mbox{ind}(T)\leq 0$  (resp., $T$ is a lower  semi-B-Fredholm with $\mbox{ind}(T)\geq 0,$ $T$ is a B-Fredholm with $\mbox{ind}(T)=0,$   $T$ is an upper semi-B-Fredholm and $p(T_{[m_{T}]})<\infty,$ $T$ is a lower  semi-B-Fredholm and $q(T_{[m_{T}]})<\infty,$  $p(T_{[m_{T}]})=q(T_{[m_{T}]})<\infty$).   If $T$ is  upper semi-B-Fredholm (resp.,  lower semi-B-Fredholm,    semi-B-Fredholm, B-Fredholm, upper semi-B-Weyl,  lower semi-B-Weyl, B-Weyl,  left Drazin invertible, right Drazin invertible, Drazin invertible) with essential degree $m_{T}=0,$ then $T$ is said to be  an upper semi-Fredholm (resp.,  lower semi-Fredholm,   semi-Fredholm, Fredholm, upper semi-Weyl, lower semi-Weyl, Weyl, upper semi-Browder, lower semi-Browder, Browder) operator.  $T$ is said to be  bounded below if $T$ is upper semi-Fredholm with $\alpha(T)=0.$
\par  The degree of stable iteration of $T$ is defined  by  $\mbox{dis}(T)=\mbox{inf}\,\Delta(T);$     where    $$\Delta(T):=\{m\in\N \,:\,  \alpha(T_{[m]})= \alpha(T_{[r]}),\,\forall r \in \N  \,  \, r\geq m \}.$$   $T$ is said to be  semi-regular if $\R(T)$ is closed and $\mbox{dis}(T)=0,$ and      $T$ is said to be  quasi-Fredholm if there exists $n \in \N$ such that  $\R(T^{n})$  is closed  and $T_{[n]}$  is semi-regular,    see   \cite{labrousse,mbekhta}. Note that every semi-B-Fredholm operator is quasi-Fredholm \cite[Proposition 2.5]{berkani-sarih}.
\par    $T$ is said to have the SVEP  at $\lambda\in\mathbb{C}$ \cite{aiena} if  for every open neighborhood $U_\lambda$ of $\lambda,$  $f\equiv 0$ is the only analytic   solution   of the equation $(T-\mu I)f(\mu)=0\quad\forall\mu\in U_\lambda.$     $T$ is said to have  the SVEP on $A\subset\mathbb{C}$ if $T$ has the SVEP at every $\lambda\in A,$  and   is said to have  the SVEP if it  has the SVEP on $\mathbb{C}.$  It is easily seen that $T\oplus S$ has the SVEP at $\lambda$ if and only if $T$ and $S$ have the SVEP at $\lambda,$ see \cite[Theorem 2.15]{aiena}.   Moreover,
\[p(T-\lambda I)<\infty \Longrightarrow \mbox{ T has the SVEP at } \lambda\,\,\,\,(A)\]
\[q(T-\lambda I)<\infty \Longrightarrow \mbox{ $T^*$ has the SVEP at } \lambda,\,\,\,\, (B)\] and    these  implications  become  equivalences  if $T-\lambda I$  has topological uniform descent \cite[Theorem 2.97, Theorem 2.98]{aiena}.  For  definitions and properties  of operators which have topological uniform descent,  see \cite{grabiner}.
 \begin{dfn}\cite{aiena} (i)  The local spectrum  of $T$ at  $x \in X$  is the set defined by $$
\sigma_{T}(x):= \left\{
    \begin{array}{ll}
       \lambda \in \C: \text{ for all open neighborhood } U_{\lambda} \text{ of } \lambda \text{ and analytic function } \\
       f: U_{\lambda} \longrightarrow X \text{ there exists } \mu \in U_{\lambda} \text{ such that } (T-\mu I)f(\mu)\neq x.
    \end{array}
\right\}
$$
(ii) If $F$ is a complex closed   subset,   then the local spectral subspace of  $T$ associated to $F$ is defined by $$X_{T}(F):=\{x\in X: \sigma_{T}(x)\subset F\}.$$
\end{dfn}
\par $T$ is nilpotent of degree $d$ if $T^{d}=0$ and $T^{d-1}\neq0$ (the null operator is nilpotent of degree 0).   $T $ is said a quasi-nilpotent   (resp., Riesz, meromorphic)  operator if $T-\lambda I$ is invertible  (resp., Browder, Drazin invertible)   for all non-zero complex  $\lambda.$ Note that $T$ is nilpotent $\Longrightarrow$ $T$ is quasi-nilpotent $\Longrightarrow$ $T$ is Riesz  $\Longrightarrow$ $T$ is meromorphic.
Denote  by $\mathcal{K}(T)$ the analytic core of $T$   (see \cite{mbekhta}):
$$\mathcal{K}(T)=\{x \in X : \exists \epsilon>0, \exists (u_{n})_{n} \subset X \text{ such that } x=u_{0}, Tu_{n+1}=u_{n} \text{ and } \|u_{n}\|\leq \epsilon^{n}\|x\| \,\,  \forall n \in \N\},$$
and by $\mathcal{H}_{0}(T)$ the  quasi-nilpotent part of $T:$ \,\,\,\, $\mathcal{H}_{0}(T)=\{x \in X : \lim\limits_{n \rightarrow \infty}\,\|T^{n}x\|^{\frac{1}{n}}=0\}.$

\medskip

 In \cite[Theorem 4]{kato}, Kato  proved that if $T$ is a semi-Fredholm operator then  $T$ is of Kato-type of degree $d,$  that's there exists $(M,N) \in \mbox{Red}(T)$ such that:
\begin{enumerate}[nolistsep]
\item[(i)]  $T_{M}$ is semi-regular.
\item[(ii)]   $T_{N}$ is nilpotent of degree $d.$
\end{enumerate}
  In a few years later  Labrousse \cite{labrousse}  characterized the Hilbert  space operators     having  Kato-type decomposition.   These    important  results of Kato and Labrousse  opened the field to many researchers to work in this direction \cite{berkani,boasso,rwassa,drazin,mbekhta,zariouh-zguitti,rwassa1,rwassa2}. In particular,  Berkani \cite{berkani} showed that $T$ is B-Fredholm (resp., B-Weyl)  if and only if there exists  $(M,N) \in \mbox{Red}(T)$ such that  $T_{M}$ is Fredholm (resp., Weyl)  and $T_{N}$ is nilpotent.  On the other hand, it is well known   \cite{drazin} that $T$ is Drazin invertible  if and only if there exists  $(M,N) \in \mbox{Red}(T)$ such that  $T_{M}$ is invertible and    $T_{N}$ is nilpotent.\\
 If     the condition (ii) ``$T_{N}$ being  nilpotent''  mentioned   in the Kato's decomposition is  replaced   by ``$T_{N}$ is quasi-nilpotent'' (resp., ``$T_{N}$ is Riesz'', ``$T_{N}$ is meromorphic''), we find  the   pseudo-Fredholm \cite{mbekhta} (resp., generalized Kato-Riesz \cite{rwassa1}, generalized Kato-meromorphic  \cite{rwassa2}) decomposition.  By the same argument  the    pseudo B-Fredholm \cite{tajmouati1,zariouh-zguitti} (resp., generalized Drazin-Riesz  Fredholm  \cite{boasso,rwassa1}, generalized Drazin-meromorphic  Fredholm \cite{rwassa2}) decomposition obtained   with    substituting   the condition    ``$T_{N}$ being  nilpotent''   in  the B-Fredholm decomposition   by   ``$T_{N}$ is quasi-nilpotent'' (resp., ``$T_{N}$ is Riesz'', ``$T_{N}$ is meromorphic'').  In the same way as  the  generalization of  B-Fredholm  and Kato decompositions,   the Drazin decomposition  was generalized \cite{koliha,rwassa1,rwassa2}.

\medskip

  We summarize in the following definition the several   known decompositions.
\begin{dfn}\cite{aznay-ouahab-zariouh2,berkani,berkani-sarih,boasso,bouamama,rwassa,mbekhta,zariouh-zguitti,rwassa1,rwassa2}  $T$ is said to be\\
(i)    Kato-type of order $d$  {\rm[}resp., quasi  upper semi-B-Fredholm,    quasi  lower semi-B-Fredholm, quasi B-Fredholm, quasi upper semi-B-Weyl, quasi lower semi-B-Weyl, quasi semi-B-Weyl{\rm]} if there exists $(M, N) \in \mbox{Red}(T)$ such that  $T_{M}$ is semi-regular {\rm[}resp., an upper  semi-Fredholm, a lower semi-Fredholm,  a Fredholm, upper semi-Weyl, lower semi-Weyl, Weyl{\rm]}   and $T_{N}$ is  nilpotent of degree $d.$  We write  $(M,N)\in KD(T)$  if it is a Kato-type decomposition.\\
(ii)   Pseudo-Fredholm {\rm[}resp., upper pseudo semi-B-Fredholm, a lower pseudo semi-B-Fredholm, a pseudo B-Fredholm, upper pseudo semi-B-Weyl, a lower pseudo semi-B-Weyl, a pseudo B-Weyl,  left generalized Drazin invertible, right generalized Drazin invertible, generalized Drazin invertible{\rm]} operator if there exists $(M, N) \in \mbox{Red}(T)$ such that  $T_{M}$ is  semi-regular {\rm[}resp., an upper  semi-Fredholm, a lower semi-Fredholm,  a Fredholm, upper semi-Weyl, lower semi-Weyl, Weyl, bounded below, surjective, invertible{\rm]}   and $T_{N}$ is  quasi-nilpotent. We write  $(M,N)\in GKD(T)$  if it is a pseudo-Fredholm decomposition.\\
(iii) Generalized Kato-Riesz    {\rm[}resp., generalized Drazin-Riesz upper semi-Fredholm,  generalized Drazin-Riesz lower semi-Fredholm, generalized Drazin-Riesz  Fredholm, generalized Drazin-Riesz upper semi-Weyl, generalized Drazin-Riesz lower semi-Weyl, generalized Drazin-Riesz  Weyl, generalized Drazin-Riesz bounded below,   generalized Drazin-Riesz surjective, generalized Drazin-Riesz invertible{\rm]} if there exists $(M, N) \in \mbox{Red}(T)$ such that  $T_{M}$ is an upper  semi-Fredholm semi-regular {\rm[}resp., an upper  semi-Fredholm, a lower semi-Fredholm,  a Fredholm, upper semi-Weyl, lower semi-Weyl, Weyl, bounded below, surjective, invertible{\rm]}    and $T_{N}$ is   Riesz.\\
(iv) Generalized Kato-meromorphic    {\rm[}resp., generalized Drazin-meromorphic upper semi-Fredholm,  generalized Drazin-meromorphic lower semi-Fredholm, generalized Drazin-meromorphic  Fredholm, generalized Drazin-meromorphic upper semi-Weyl, generalized Drazin-meromorphic lower semi-Weyl, generalized Drazin-meromorphic  Weyl, generalized Drazin-meromorphic bounded below,   generalized Drazin-meromorphic surjective, generalized Drazin-meromorphic invertible{\rm]} if there exists $(M, N) \in \mbox{Red}(T)$ such that  $T_{M}$ is  semi-regular {\rm[}resp.,  upper  semi-Fredholm,  lower semi-Fredholm,  Fredholm, upper semi-Weyl, lower semi-Weyl, Weyl, bounded below, surjective, invertible{\rm]}   and $T_{N}$ is   meromorphic.
\end{dfn}
 As  a continuation of the previous studies mentioned above,  we define  new classes of operators, one of these classes   named $g_{z}$-Kato (see Definition \ref{dfn1.3}) which  generalizes  the   generalized Kato-meromorphic operators class. We prove  that the $g_{z}$-Kato spectrum $\sigma_{g_{z}K}(T)$ is compact,   $\sigma_{g_{z}K}(T)\subset \mbox{acc}\,\sigma_{pf}(T).$   Moreover, we show in Proposition \ref{prop1.6}  that  if $T$ is    $g_{z}$-Kato then  $\alpha(T_{M}),$ $\beta(T_{M}),$ $p(T_{M})$ and $q(T_{M})$ are  independent of  the decomposition  $(M,N) \in g_{z}KD(T).$ The other class named $g_{z}$-invertible  generalizes the class of generalized Drazin invertible operators introduced by Koliha.   We prove  that   $T$ is  $g_{z}$-invertible  if and only if $0\not\in \mbox{acc}\,(\mbox{acc}\,\sigma(T))$ and we prove   in turn that this  is equivalent  to say that there exists  a Drazin invertible operator $S$  such that $TS = ST,$ $STS = S$ and  $T^{2}S - T$ is zeroloid.  These  characterizations  are analogous  to the one proved by Koliha \cite{koliha} which established that $T$ is   generalized Drazin invertible operator    if and only if $0\not\in \mbox{acc}\,\sigma(T)$ if and only if there exists    an  operator $S$  such that $T S = ST,$ $STS = S$ and  $T^{2}S - T$ is quasi-nilpotent.  As applications, using the   new spectra studied in present work  and the concept of the Weak SVEP introduced   at the end  of  this paper, we  give new  characterizations  of Browder-type theorems   (Theorem \ref{thm.j1} and Proposition \ref{propj.3}).

\medskip

The  next  list summarizes some notations and symbols that we will need later.
$$
    \begin{array}{ll}
  \noindent r(T) &:   \text{ the   spectral radius of  } T\\
\noindent \mbox{iso}\,A  &: \text{  isolated points of  a complex subset   } A\\
\noindent \mbox{acc}\,A &: \text{  accumulation points of a complex  subset } A\\
 \noindent \overline{A} &: \text{  the closure of a   complex subset } A\\
 \noindent A^C &: \text{ the complementary of a complex subset } A\\
 \noindent B(\lambda, \epsilon) &: \text{ the  open ball of radius } \epsilon  \text{ centered at } \lambda \\
 \noindent D(\lambda, \epsilon) &: \text{ the  closed  ball of radius } \epsilon  \text{ centered at } \lambda\\
\noindent (B) &: \text{ the class of operators satisfying Browder's theorem }  \, \left(T\in(B) \text{ if }  \sigma_{w}(T)=\sigma_{b}(T)\right)\\
\noindent (B_{e}) &: \text{ the class of operators satisfying  essential Browder's theorem } \cite{aznay-ouahab-zariouh}\,\,\,   \left(T\in(B_{e}) \text{ if } \sigma_{e}(T)=\sigma_{b}(T)\right)\\
\noindent (aB) &: \text{ the class of operators satisfying a-Browder's theorem }  \, \left(T\in(aB) \text{ if }  \sigma_{uw}(T)=\sigma_{ub}(T)\right)\\
    \end{array}
$$
 \begin{tabular}{l|l}

 $\sigma(T)$:   spectrum of $T$ &  $\sigma_{pf}(T)$:  pseudo-Fredholm spectrum of $T$\\
 $\sigma_{a}(T)$:  approximate points spectrum of $T$ & $\sigma_{pbf}(T)$:  pseudo B-Fredholm spectrum of $T$\\
 $\sigma_{s}(T)$: surjective  spectrum of $T$ & $\sigma_{upbf}(T)$:  upper pseudo semi-B-Fredholm spectrum of $T$\\
 $\sigma_{se}(T)$:  semi-regular spectrum of $T$ &  $\sigma_{lpbf}(T)$:  lower  pseudo semi-B-Fredholm spectrum of $T$\\
 $\sigma_{e}(T)$:  essential  spectrum of $T$ &  $\sigma_{pbw}(T)$:   pseudo B-Weyl spectrum of $T$\\
 $\sigma_{uf}(T)$:  upper semi-Fredholm spectrum  of $T$ & $\sigma_{upbw}(T)$:  upper pseudo semi-B-Weyl spectrum of $T$\\
 $\sigma_{lf}(T)$: lower  semi-Fredholm spectrum  of $T$ & $\sigma_{lpbw}(T)$:  lower pseudo semi-B-Weyl spectrum of $T$\\
 $\sigma_{w}(T)$: Weyl spectrum of $T$ & $\sigma_{gd}(T)$:   generalized Drazin invertible  spectrum of $T$\\
 $\sigma_{uw}(T)$: upper semi-Weyl spectrum of $T$ &  $\sigma_{lgd}(T)$:   left generalized Drazin invertible  spectrum of $T$\\
 $ \sigma_{lw}(T)$:   lower  semi-Weyl  spectrum of $T$ &  $\sigma_{rgd}(T)$:   right generalized Drazin invertible  spectrum of $T$\\
$\sigma_{b}(T)$:  Browder spectrum of $T$ &   $\sigma_{d}(T)$:  Drazin spectrum of $T$\\
$\sigma_{bf}(T)$: B-Fredholm  spectrum of $T$ & $\sigma_{bw}(T)$: B-Weyl spectrum of $T$
\end{tabular}

\section{$g_{z}$-Kato decomposition}
We begin this section by the following definition of zeroloid operators.
\begin{dfn}\label{dfn1.1} We say that $T \in L(X)$ is an operator zeroloid  if $\mbox{acc}\,\sigma(T)\subset \{0\}.$
\end{dfn}
\vspace{0.002em}

The next  remark summarizes  some properties of zeroloid operators.
\begin{rema}\label{rema1.2}(i) A zeroloid operator has at most a countable spectrum.\\
 (ii)  It is well known that    $\mbox{acc}\,\sigma(T)\subset\sigma_{d}(T)$ for every   $T \in L(X).$  Thus  every meromorphic operator   is zeroloid. But  the  operator $I+Q$ shows that the converse is not true; where $I$ is the identity operator  and  $Q$ is  the quasi-nilpotent  operator defined on  the Hilbert space $\ell^{2}(\N)$ by $Q(x_{1}, x_{2},\dots)=(0, x_{1}, \frac{x_{2}}{2},\dots).$\\
(iii) $T$ is zeroloid if and only if $T^{n}$ is zeroloid for every  integer $n \geq 1.$\\
(iv) Let $(T, S) \in L(X)\times L(Y);$ where  $Y$ is also a Banach space. Then  $T\oplus S$ is zeroloid if and only if $T$  and  $S$ are zeroloid.\\
(v)  Here and elsewhere denote by  $\mbox{comm}(T):=\{S \in L(X): TS=ST\}.$ So  if $Q\in \mbox{comm}(T)$   is a  quasi-nilpotent  or a   power finite rank operator,  then $T$ is zeroloid  if and only if    $T+Q$ is zeroloid.
\end{rema}
 According to \cite{aznay-ouahab-zariouh}, the ascent $\tilde{p}(T)$ and the descent $\tilde{q}(T)$ [noted in \cite{aznay-ouahab-zariouh} by $p(T)$ and $q(T)$] of  a pseudo-Fredholm operator $T\in L(X)$ are defined respectively, by  $\tilde{p}(T)=p(T_{M})$ and  $\tilde{q}(T)=q(T_{M});$  where $M$ is any subspace which complemented by a subspace  $N$ such that $(M,N) \in GKD(T).$
\begin{prop}\label{propa1} Let $T \in L(X)$ be a pseudo-Fredholm operator.  The following statements are equivalent.\\
(a) $\tilde{p}(T)<\infty;$\\
(b) $T$ has the SVEP at $0;$\\
(c) $\mathcal{H}_{0}(T)\cap \mathcal{K}(T)=\{0\};$\\
(d)  $\mathcal{H}_{0}(T)$ is closed.\\
dually, the following are equivalent.\\
(e) $\tilde{q}(T)<\infty;$\\
(f) $T^{*}$ has the SVEP at $0;$\\
(g) $\mathcal{H}_{0}(T)+ \mathcal{K}(T)=X.$
\end{prop}
\begin{proof}
(a) $\Longleftrightarrow$ (b)  Let $(M,N) \in GKD(T),$ then  $T_{M}$ is semi-regular and $T_{N}$ is quasi-nilpotent. As  $p(T_{M})=\tilde{p}(T)$ then by the implication (A)  above, we deduce that $\tilde{p}(T)<\infty$  if and only if   $T_{M}$ has the SVEP at 0. Hence $\tilde{p}(T)<\infty$ if and only if $T$ has the SVEP at 0. The equivalence (e) $\Longleftrightarrow$ (f) goes similarly. The equivalences  (b) $\Longleftrightarrow$ (c),  (c) $\Longleftrightarrow$ (d) and  (f) $\Longleftrightarrow$ (g)  are proved in  \cite[Theorem 2.79, Theorem 2.80]{aiena}.
\end{proof}
\begin{lem} Let $T \in L(X).$ The following are equivalent.\\
(i) $T$ is zeroloid;\\
(ii) $\sigma_{*}(T)\subset \{0\};$ where $\sigma_{*}\in \{\sigma_{pf},\sigma_{upbf}, \sigma_{lpbf},\sigma_{upbw}, \sigma_{lpbw}, \sigma_{lgd}, \sigma_{rgd}, \sigma_{pbf}, \sigma_{pbw}\}.$
\end{lem}
\begin{proof}  (i) $\Longrightarrow$ (ii)  Obvious, since $\sigma_{gd}(T)=\mbox{acc}\,\sigma(T).$\\
 (ii) $\Longrightarrow$ (i) Assume that $\sigma_{*}(T)\subset \{0\}.$ Then $\C\setminus\{0\}\subset \Omega;$ where $\Omega$  is the  component of $(\sigma_{pf}(T))^{C}.$  Suppose that there exists $\lambda \in  \C\setminus\{0\}$   such that $\lambda \in \mbox{acc}\,\sigma(T),$  then $\lambda \notin \sigma_{*}(T).$ Hence $\tilde{p}(T-\lambda I)=\infty$ or $\tilde{q}(T-\lambda I)=\infty.$  Suppose that   $\tilde{p}(T-\lambda I)=\infty,$ as $T-\lambda I$ is  pseudo-Fredholm, from Proposition \ref{propa1} we have  $\mathcal{H}_{0}(T-\lambda I)\cap \mathcal{K}(T-\lambda I)\neq\{0\}.$ And  from \cite[Corollary 4.3]{bouamama}   we obtain   $\overline{\mathcal{H}_{0}(T-\lambda I)}\cap \mathcal{K}(T-\lambda I)=\overline{\mathcal{H}_{0}(T-\mu I)}\cap \mathcal{K}(T-\mu I)$ for every  $\mu\in \Omega.$ If there exists $\mu\in \Omega\setminus\{0\}$ such that  $\mathcal{H}_{0}(T-\mu I)\cap \mathcal{K}(T-\mu I)=\{0\}$ then $\mathcal{H}_{0}(T-\mu I)$ is closed.  So $\overline{\mathcal{H}_{0}(T-\lambda I)}\cap \mathcal{K}(T-\lambda I)=\{0\}$ and this is a contradiction.  The case of $\tilde{q}(T-\lambda I)=\infty$ goes similarly. Hence $T$ is zeroloid.
\end{proof}
\begin{prop}\label{propreszeroloid} $T \in L(X)$  is zeroloid if and only if $T_{M}$ and $T^{*}_{M^{\perp}}$  are zeroloid; where $M$ is any closed  $T$-invariant subspace.
\end{prop}
\begin{proof} If     $T$ is zeroloid then its   resolvent $(\sigma(T))^{C}$ is  connected. And from    \cite[Proposition 2.10]{djordjevec-duggal},  we obtain that   $\sigma(T)=\sigma(T_{M})\cup\sigma(T^{*}_{M^{\perp}}).$ Thus  $T_{M}$ and $T^{*}_{M^{\perp}}$  are zeroloid.  Conversely, if   $T_{M}$ and $T^{*}_{M^{\perp}}$  are zeroloid   then   $T$  is zeroloid, since the inclusion $\sigma(T)\subset\sigma(T_{M})\cup\sigma(T^{*}_{M^{\perp}})$ is always true.
\end{proof}

\begin{dfn}\label{dfn1.3} Let $T \in L(X).$  A pair of subspaces $(M,N) \in \mbox{Red}(T)$ is   a generalized Kato zeroloid  decomposition  associated to $T$ ($(M,N) \in g_{z}KD(T)$ for brevity)  if   $T_{M}$ is  semi-regular   and $T_{N}$ is  zeroloid.  If such a pair  exists,    we say that  $T$ is   a $g_{z}$-Kato operator.
\end{dfn}

\begin{ex}\label{ex1.4}   \noindent (i)  A zeroloid operator  and a semi-regular operator   are $g_{z}$-Kato.~~

\noindent (ii) A  generalized Kato-meromorphic   operator  is  $g_{z}$-Kato. But the converse is not true, see Example \ref{exs} below.
\end{ex}
Our  next result gives a punctured neighborhood theorem for $g_{z}$-Kato operators. Recall that  the reduced minimal modulus $\gamma(T)$ of an operator $T$ is defined by $\gamma(T):=\underset{x\notin \mathcal{N}(T)}{\mbox{inf}}\,\frac{\|Tx\|}{d(x,\mathcal{N}(T))};$ where $d(x,\mathcal{N}(T))$ is the distance between $x$ and $\mathcal{N}(T).$
\begin{thm}\label{thm1.5} Let   $T \in L(X)$ be  a $g_{z}$-Kato operator.       For every  $(M,N) \in g_{z}KD(T),$ there   exists $\epsilon>0$ such that  for all  $\lambda \in B(0, \epsilon)\setminus\{0\}$ we have\\
(i) $T-\lambda I$ is  pseudo-Fredholm.\\
(ii) $\alpha(T_{M})=\mbox{dim}\,\mathcal{N}(T-\lambda I)\cap \mathcal{K}(T-\lambda I)\leq \alpha(T-\lambda I).$\\
(iii)   $\beta(T_{M})=\mbox{codim}\,\R(T-\lambda I)+\mathcal{H}_{0}(T-\lambda I)\leq \beta(T-\lambda I).$
\end{thm}
\begin{proof}  Let $\epsilon=\gamma(T_{M})>0$  and let $\lambda \in B(0, \epsilon)\setminus\{0\}.$  Then   \cite[Theorem 4.7]{grabiner} implies that the operator  $T_{M}-\lambda I $ is  semi-regular, $\alpha(T_{M})=\alpha(T_{M}-\lambda I)$ and $\beta(T_{M})=\beta(T_{M}-\lambda I).$  On the other hand, as  $T_{N}$ is  zeroloid then from  \cite{aznay-ouahab-zariouh}, $T_{N}-\lambda  I$  is  pseudo-Fredholm  with  $\mathcal{N}(T_{N}-\lambda I)\cap \mathcal{K}(T_{N}-\lambda I)=\{0\}$ and $N=\R(T_{N}-\lambda I)+\mathcal{H}_{0}(T_{N}-\lambda I).$     Hence  $T-\lambda  I$ is  pseudo-Fredholm, $\alpha(T_{M})=\mbox{dim}\,\mathcal{N}(T-\lambda I)\cap \mathcal{K}(T-\lambda I)$ and $\beta(T_{M})=\mbox{codim}\,\R(T-\lambda I)+\mathcal{H}_{0}(T-\lambda I).$
\end{proof}
Since every pseudo-Fredholm operator is $g_{z}$-Kato,   from Theorem \ref{thm1.5} we immediately obtain  the following corollary. Hereafter, we denote by $\sigma_{g_{z}K}(T)=\{\lambda \in \C: T-\lambda I \text{ is not }g_{z}\text{-Kato} \text{ operator}\}$ the  $g_{z}$-Kato spectrum.
\begin{cor}   The $g_{z}$-Kato spectrum $\sigma_{g_{z}K}(T)$ of an operator $T \in L(X)$  is a compact subset of $\C.$
\end{cor}
\begin{prop}\label{prop1.6}   If  $T \in L(X)$ is  a $g_{z}$-Kato operator then $\alpha(T_{M}),$  $\beta(T_{M}),$ $p(T_{M})$ and   $q(T_{M})$ are independent of the choice of the generalized Kato zeroloid decomposition $(M,N)\in g_{z}KD(T).$
\end{prop}
\begin{proof}  Let $(M_{1},N_{1}),(M_{2},N_{2})  \in g_{z}KD(T)$   and  let  $n \in \N\setminus\{0\}.$  It is easily seen that $T^{n}$ is also   a $g_{z}$-Kato operator and   $(M_{1},N_{1}),(M_{2},N_{2}) \in g_{z}KD(T^{n}).$ We put $\epsilon_{n}:=\mbox{min}\{\gamma(T^{n}_{M_{1}}),\gamma(T^{n}_{M_{2}})\}.$ Let  $\lambda \in  B(0, \epsilon_{n}) \setminus\{0\},$     by   Theorem \ref{thm1.5}  we obtain      $\alpha(T^{n}_{M_{1}})=\alpha(T^{n}_{M_{2}})=\mbox{dim}\,\mathcal{N}(T^{n}-\lambda I)\cap \mathcal{K}(T^{n}-\lambda I)$  and $\beta(T^{n}_{M_{1}})=\beta(T^{n}_{M_{2}})=\mbox{codim}\,\R(T^{n}-\lambda I)+\mathcal{H}_{0}(T^{n}-\lambda I).$  Hence $p(T_{M_{1}})=p(T_{M_{2}})$ and $q(T_{M_{1}})=q(T_{M_{2}}).$
\end{proof}
Let $T \in L(X)$ be   a $g_{z}$-Kato operator. Following Proposition \ref{prop1.6},  we denote by   $\tilde{\alpha}(T)=\alpha(T_{M}),$  $\tilde{\beta}(T)=\beta(T_{M}),$   $\tilde{p}(T)=p(T_{M})$ and  $\tilde{q}(T)=q(T_{M});$ where $(M,N) \in g_{z}KD(T)$ be arbitrary. Furthermore,  if   $T_{M}$ is semi-Fredholm then for every $(M^{'}, N^{'}) \in g_{z}KD(T)$ the operator $T_{M^{'}}$ is also semi-Fredholm  and $\mbox{ind}(T_{M})=\mbox{ind}(T_{M^{'}})$  (this result will be extended in  Lemma \ref{lem2.4}).
\vspace{0.2cm}
\par  In the sequel, for $T \in L(X)$ and $(M,N)\in \mbox{Red}(T),$  we define the operator    $T_{(M,N)}\in L(X)$  by $T_{(M,N)}:=TP_{M}+ P_{N};$ where $P_{M}$  is the projection operator on $X$ onto $M.$ The next  lemma   extends \cite[Theorem A.16]{muller}.
\begin{lem}\label{lem1.7} Let $T \in L(X)$ and let  $(M,N) \in \mbox{Red}(T).$  The following  assertions are equivalent.\\
(i) $\R(T_{M})$ is closed;\\
(ii) $\R(T^{*}_{N^{\perp}})$ is closed;\\
(iii) $\R(T^{*}_{N^{\perp}})\oplus M^{\perp}$ is closed in  the weak-*-topology  $\sigma(X^{*}, X)$ on $X^{*}.$
\end{lem}
\begin{proof} As $(M,N) \in \mbox{Red}(T)$ then $(P_{N})^{*}=P_{M^{\perp}}$ and $(TP_{M})^{*}=T^{*}P_{N^{\perp}}.$     Then $(T_{(M,N)})^{*}=(TP_{M}+ P_{N})^{*}=T^{*}P_{N^{\perp}}+ P_{M^{\perp}}=T^{*}_{(N^{\perp},M^{\perp})}.$ Thus $\R(T_{(M,N)})=\R(T_{M})\oplus N$ and  $\R((T_{(M,N)})^{*})=\R(T^{*}_{N^{\perp}})\oplus M^{\perp}.$  Remark that  $\R(T_{M})$ is closed if and only if $\R(T_{(M,N)})$ is closed. Hence, the  proof  is complete  by applying  \cite[Theorem A.16]{muller}  to the operator $T_{(M,N)}.$
\end{proof}
Using Lemma \ref{lem1.7} and some well known  classical properties of pseudo-Fredholm and  quasi-Fredholm operators, we immediately obtain:
\begin{cor} Let $T \in L(X).$ The following statements hold.\\
(i) If $T$ is pseudo-Fredholm, then  $\R(T^{*})+\mathcal{H}_{0}(T^{*})$   is closed in  $\sigma(X^{*}, X).$\\
(ii) If $T$ is a Hilbert space quasi-Fredholm operator  of degree $d,$ then  $\R(T^{*})+\mathcal{N}(T^{d*})$ is   closed in  $\sigma(X^{*}, X).$
\end{cor}
\par The following lemma extends some  well known results in spectral theory, as   relation between nullity, deficiency   and    some other spectral quantities of a given operator $T$ and its dual  $T^{*}.$
\begin{lem}\label{lem1.8}  Let $T \in L(X)$ and let  $(M,N) \in \mbox{Red}(T).$   the following statements hold.\\
(i) $T_{M}$ is semi-regular if and only if $T^{*}_{N^{\perp}}$ is semi-regular.\\
(ii) If $\R(T_{M})$ is closed then   $\alpha(T_{M})=\beta(T^{*}_{N^{\perp}}),$  $\beta(T_{M})=\alpha(T^{*}_{N^{\perp}}),$ $p(T_{M})=q(T^{*}_{N^{\perp}})$ and $q(T_{M})=p(T^{*}_{N^{\perp}}).$\\
(iii) $\sigma_{a}(T_{M})=\sigma_{s}(T^{*}_{N^{\perp}}),$ $\sigma_{s}(T_{M})=\sigma_{a}(T^{*}_{N^{\perp}}),$   $\sigma_{*}(T_{M})=\sigma_{*}(T^{*}_{N^{\perp}})$ and $r(T_{M})=r(T^{*}_{N^{\perp}});$ where $\sigma_{*}\in \{ \sigma, \sigma_{se},  \sigma_{e},\sigma_{sf}, \sigma_{bf}, \sigma_{d}, \sigma_{b}\}.$ And  if $T_{M}$ is semi-Fredholm then $\mbox{ind}(T_{M})=-\mbox{ind}(T^{*}_{N^{\perp}}).$
\end{lem}
\begin{proof}
(i) We have   $\mathcal{N}(T_{(M,N)})=\mathcal{N}(T_{M})$ and  $(T_{(M,N)})^{n}=T^{n}_{(M,N)}$ for every $n\in \N.$ It is easy to see that  $T_{M}$ is semi-regular if and only if $T_{(M,N)}$ is semi-regular. As $(T_{(M,N)})^{*}=T^{*}_{(N^{\perp},M^{\perp})}$  then $T_{M}$ is semi-regular if and only if $T^{*}_{N^{\perp}}$ is semi-regular.\\
(ii) We have       $\mathcal{N}((T_{(M,N)})^{n})=\mathcal{N}(T^{n}_{M})$ and  $\R((T_{(M,N)})^{n})=\R(T^{n}_{M})\oplus N$  for every $n\in \N.$ As $\R(T_{(M,N)})=\R(T_{M})\oplus N$ is closed then   $\alpha(T_{M})=\alpha(T_{(M,N)})=\beta(T^{*}_{(N^{\perp},M^{\perp})})=\beta(T^{*}_{N^{\perp}}).$ The other equalities go similarly.\\
(iii)  As $(T_{M}\oplus 0_{N})^{*}=(TP_{M})^{*}=T^{*}P_{N^{\perp}}=T^{*}_{N^{\perp}}\oplus 0_{M^{\perp}},$ then $\sigma_{*}(T_{M})\cup\sigma_{*}(0_{N})=\sigma_{*}(T_{M}\oplus 0_{N})=\sigma_{*}(T^{*}_{N^{\perp}}\oplus 0_{M^{\perp}})=\sigma_{*}(T^{*}_{N^{\perp}})\cup\sigma_{*}(0_{M^{\perp}}).$ We  know   that   $\sigma_{*}(S)=\emptyset$ for every nilpotent   operator $S$ with   $\sigma_{*}\in \{\sigma_{bf}, \sigma_{d}\}.$  Moreover,  the first and the second points imply that  $0\in \sigma_{*}(T_{M})$ if and only if  $0 \in\sigma_{*}(T^{*}_{N^{\perp}});$ where  $\sigma_{*}\in \{ \sigma, \sigma_{se},  \sigma_{e},\sigma_{sf}, \sigma_{b}\}.$ So    $\sigma_{*}(T_{M})=\sigma_{*}(T^{*}_{N^{\perp}})$ and $r(T_{M})=r(T^{*}_{N^{\perp}}).$  The proof of other equalities  spectra is obvious, see Lemma \ref{lem1.7}. Moreover,  if $T_{M}$ is semi-Fredholm then  $T^{*}_{N^{\perp}}$ is also  semi-Fredholm  and $\mbox{ind}(T_{M})=-\mbox{ind}(T^{*}_{N^{\perp}}).$
\end{proof}
\begin{cor}\label{cor1.9} Let $T \in L(X)$ and let  $(M,N) \in \mbox{Red}(T).$  Then    $(M, N)\in  g_{z}KD(T)$ if and only if   $(N^{\perp}, M^{\perp})\in  g_{z}KD(T^{*}).$ In particular, if $T$ is   $g_{z}$-Kato  then $T^{*}$ is   $g_{z}$-Kato.
\end{cor}
\begin{prop}\label{prop1.10} If $T\in L(X)$ is  $g_{z}$-Kato then\\
(a) There exist $S,R \in L(X)$ such that: ~~

 (i) $T=S + R,$ $RT=TR=0,$ $S$ is quasi-Fredholm of degree $d\leq 1$ and   $R$ is  zeroloid. ~~

(ii)  $\mathcal{N}(S)+ \mathcal{N}(R)=X$ and  $\R(S)\oplus  \overline {\R(R)}$ is closed.\\
(b) There exist $S,R \in L(X)$ such that $SR=RS=(S+R)-I=T,$ $S$ is semi-regular   and $R$ is zeroloid.
\end{prop}
\begin{proof} (a) Let $(M,N) \in g_{z}KD(T).$ Then the operators $S=TP_{M}$ and $R=TP_{N}$ respond to the statement (a). Indeed, as $T_{N}$ is zeroloid and  $\mbox{acc}\,\sigma(R)=\mbox{acc}\,\sigma(T_{N})$ then $R$ is zeroloid. Suppose that  $M\notin \{\{0\}, X\}$  (the other case is trivial) and let $n \in \N\setminus\{0\},$ then   $\mathcal{N}(S^{n})=N \oplus \mathcal{N}(T^{n}_{M})$ and $\R(S)=\R(T_{M})$ is closed. Therefore $\mathcal{N}(S^{n})+\R(S)=N + \mathcal{N}(T^{n}_{M})+\R(T_{M})=N+\mathcal{N}(T_{M})+\R(T_{M})=\mathcal{N}(S)+\R(S),$ since $T_{M}$ is semi-regular. Consequently, $S$ is  quasi-Fredholm  of degree $d\leq 1.$ Moreover, $\mathcal{N}(S)+ \mathcal{N}(R)=X$ and  $\R(S)\oplus  \overline {\R(R)}=\R(T_{M})\oplus  \overline {\R(T_{N})}$ is closed.\\
(ii) Let   $(M,N) \in g_{z}KD(T).$  If we take  $S=T_{(M,N)}$ and  $R=T_{(N,M)}$  then $SR=RS=(S+R)-I=T,$   $S=T_{M}\oplus I_{N}$  is semi-regular and $R= I_{M}\oplus T_{N}$ is zeroloid.
\end{proof}
In the case of Hilbert space operator $T,$  the next proposition  shows that   the statement $(a)$ of  Proposition \ref{prop1.10} is equivalent to  say that    $T$ is $g_{z}$-Kato.
\begin{prop}\label{thm1.11} If $H$ is a Hilbert space, then  $T\in L(H)$  is    $g_{z}$-Kato  if and only if there exist $S,R \in L(H)$ such that $T=S + R$ and\\
(i) $RT=TR=0,$ $S$ is quasi-Fredholm of degree $\mbox{dis}(S)\leq 1,$  $R$ is a zeroloid operator;\\
(ii) $\mathcal{N}(S)+ \mathcal{N}(R)=H$ and  $\R(S)\oplus  \overline {\R(R)}$ is closed.
\end{prop}
\begin{proof} Assume that $S$ is quasi-Fredholm of degree 1 (the case of  $S$ is semi-regular is obvious) then  from the proof of \cite[Theorem 2.2]{mbekhta},  there exists  $(M,N) \in GKD(S)$ such that $T_{M}=S_{M}$ and  $T_{N}=R_{N}.$  As $R$ is zeroloid then Proposition \ref{propreszeroloid} entails that $T_{N}$ is zeroloid. Thus $T$ is $g_{z}$-Kato. For  the converse, see Proposition \ref{prop1.10}.
\end{proof}

\section{$g_{z}$-Fredholm operators}
\begin{dfn}\label{dfn2.1} $T \in L(X)$ is said  an upper semi-$g_{z}$-Fredholm (resp., a lower semi-$g_{z}$-Fredholm, a $g_{z}$-Fredholm)  operator    if there exists $(M, N) \in \mbox{Red}(T)$ such that  $T_{M}$ is an upper  semi-Fredholm (resp., a lower semi-Fredholm, a Fredholm) operator  and $T_{N}$ is  zeroloid.  $T$ is said a semi-$g_{z}$-Fredholm   if it is an  upper or  a lower semi-$g_{z}$-Fredholm.
\end{dfn}
\noindent As examples, every  generalized Drazin-meromorphic  semi-Fredholm    is a semi-$g_{z}$-Fredholm. But we show in  Example \ref{exs} (see later) that the converse is not true.   And   every zeroloid operator is  $g_{z}$-Fredholm.

\vspace{0.2cm}
The next proposition gives some  relations between  semi-$g_{z}$-Fredholm  and   $g_{z}$-Kato operators.
\begin{prop}\label{prop2.3}
  Let $T \in L(X).$ The following statements are equivalent.\\
(i)  $T$ is  semi-$g_{z}$-Fredholm {\rm[}resp.,   upper semi-$g_{z}$-Fredholm, lower semi-$g_{z}$-Fredholm, $g_{z}$-Fredholm{\rm]};\\
(ii) $T$  is $g_{z}$-Kato and   $\mbox{min}\,\{\tilde{\alpha}(T), \tilde{\beta}(T)\}<\infty$  {\rm[}resp.,    $T$  is $g_{z}$-Kato and  $\tilde{\alpha}(T)<\infty,$   $T$  is $g_{z}$-Kato and  $\tilde{\beta}(T)<\infty,$ $T$  is $g_{z}$-Kato and   $\mbox{max}\,\{\tilde{\alpha}(T), \tilde{\beta}(T)\}<\infty${\rm]};\\
(iii)  $T$  is $g_{z}$-Kato and  $0 \notin \mbox{acc}\,\sigma_{spbf}(T)$  {\rm[}resp.,   $T$  is $g_{z}$-Kato and $0 \notin \mbox{acc}\,\sigma_{upbf}(T),$ $T$  is $g_{z}$-Kato and $0 \notin \mbox{acc}\,\sigma_{lpbf}(T),$  $T$  is $g_{z}$-Kato and $0 \notin \mbox{acc}\,\sigma_{pbf}(T)${\rm]}.
\end{prop}
\begin{proof} $(i) \Longleftrightarrow (ii)$ Assume  that $T$ is semi-$g_{z}$-Fredholm, then  there exists  $(A,B)\in \mbox{Red}(T)$ such that $T_{A}$ is semi-Fredholm and $T_{B}$ is zeroloid.  From \cite[Corollary  3.7]{aznay-ouahab-zariouh2}, there exists a pair  $(M, N)\in g_{z}KD(T)$ such that    $T_{M}$ is semi-Fredholm.  Thus $T$ is  $g_{z}$-Kato operator and  $\mbox{min}\,\{\tilde{\alpha}(T), \tilde{\beta}(T)\}=\mbox{min}\,\{\alpha(T_{M}), \beta(T_{M})\}<\infty.$ The converse is obvious.  The other    equivalence cases  go similarly.\\
$(ii) \Longleftrightarrow (iii)$ Is a consequence of Theorem \ref{thm1.5}.
\end{proof}
\begin{cor} $T\in L(X)$  is   $g_{z}$-Fredholm    if and only if $T$ is  an  upper  and    lower  semi-$g_{z}$-Fredholm.
\end{cor}
The following lemma  will allow us  to define the   index for    semi-$g_{z}$-Fredholm operators.
\begin{lem}\label{lem2.4} Let $T \in L(X).$ If   there exist two pair  of closed $T$-invariant subspaces $(M, N)$ and $(M^{'}, N^{'})$ such that   $M \oplus  N=M^{'}\oplus N^{'}$ is closed, $T_{M}$ and   $T_{M^{'}}$ are semi-Fredholm operators, $T_{N}$ and   $T_{N^{'}}$ are zeroloid operators, then $\mbox{ind}(T_{M})= \mbox{ind}(T_{M^{'}}).$
\end{lem}
\begin{proof}
As  $T_{M}$ and  $T_{M^{'}}$ are  semi-Fredholm operators   then from punctured neighborhood theorem for semi-Fredholm operators, there  exists $\epsilon >0$ such that $B(0, \epsilon) \subset \sigma_{sf}(T_{M})^C\cap \sigma_{sf}(T_{M^{'}})^C,$  $\mbox{ind}(T_{M}- \lambda I) =\mbox{ind}(T_{M})$ and   $\mbox{ind}(T_{M^{'}}- \lambda I) =\mbox{ind}(T_{M^{'}})$ for every $\lambda \in B(0, \epsilon).$    From \cite[Remark 2.4]{aznay-ouahab-zariouh} and the fact that $T_{N}$ and  $T_{N^{'}}$ are  zeroloid,  we conclude that $B_{0}:=B(0, \epsilon)\setminus\{0\}\subset  \sigma_{sf}(T_{M})^C\cap \sigma_{sf}(T_{M^{'}})^C\cap  \sigma_{gd}(T_{N})^C\cap \sigma_{gd}(T_{N^{'}})^C\subset\sigma_{spbf}(T)^C.$  Let $\lambda \in B_{0},$  then $T-\lambda I$ is   pseudo semi-B-Fredholm and      $\mbox{ind}(T- \lambda I)=\mbox{ind}(T_{M}- \lambda I)+\mbox{ind}(T_{N} - \lambda I)=\mbox{ind}(T_{M^{'}}- \lambda I)+\mbox{ind}(T_{N^{'}} - \lambda I).$ Thus  $\mbox{ind}(T_{M})=\mbox{ind}(T_{M^{'}}).$
\end{proof}

\begin{dfn}  Let  $T\in L(X)$ be  a semi-$g_{z}$-Fredholm. We  define its   index $\mbox{ind}(T)$ as the index of  $T_{M};$ where $M$ is a closed $T$-invariant subspace  which  has a complementary closed $T$-invariant subspace  $N$ such that $T_{M}$ is    semi-Fredholm  and $T_{N}$ is zeroloid. From Lemma \ref{lem2.4}, the index of $T$ is independent of the choice of the pair  $(M, N)$ appearing in  Definition \ref{dfn2.1} of $T$ as a   semi-$g_{z}$-Fredholm. In addition,  from  Proposition \ref{prop2.3}  we have  $\mbox{ind}(T)=\tilde{\alpha}(T)-\tilde{\beta}(T).$
\end{dfn}
We say that $T \in L(X)$ is an   upper semi-$g_{z}$-Weyl (resp., a lower  semi-$g_{z}$-Weyl, a  semi-$g_{z}$-Weyl) operator   if   $T$  is an upper semi-$g_{z}$-Fredholm  (resp., a lower  semi-$g_{z}$-Fredholm, a  $g_{z}$-Fredholm) with $\mbox{ind} (T)\leq 0$ (resp.,  $\mbox{ind} (T)\geq 0$,  $\mbox{ind} (T)=0$).
\begin{rema}\label{rema2.5} (i) Every zeroloid operator $T$  is a $g_{z}$-Fredholm with $\tilde{\alpha}(T)=\tilde{\beta}(T)=\mbox{ind}(T)=0.$  A pseudo semi-B-Fredholm   is      semi-$g_{z}$-Fredholm and  its usual  index   coincides with its index as a  semi-$g_{z}$-Fredholm.\\
(ii)  $T$  is    $g_{z}$-Fredholm   if and only if $T$ is    semi-$g_{z}$-Fredholm with an integer index  $\mbox{ind}(T).$
 $T$  is   $g_{z}$-Weyl   if and only if $T$  is an  upper and     lower  semi-$g_{z}$-Weyl.
\end{rema}
\begin{prop}\label{prop2.6}   Let  $T \in L(X)$ and    $S \in L(Y)$  are      semi-$g_{z}$-Fredholm.  Then \\
(i) $T^{n}$ is semi-$g_{z}$-Fredholm and $\mbox{ind}(T^{n})=n.\mbox{ind}(T),$ for every integer $n\geq 1.$\\
(ii) $T\oplus S$ is  semi-$g_{z}$-Fredholm and   $\mbox{ind} (T\oplus S)=\mbox{ind} (T)+ \mbox{ind} (S).$
\end{prop}
\begin{proof}(i)  As $T$  is semi-$g_{z}$-Fredholm, then  there exists $(M, N) \in \mbox{Red}(T)$ such that $T_{M}$ is   semi-Fredholm   and $T_{N}$ is zeroloid. So   $(M, N) \in Red(T^{n}),$  $T_{M}^{n}$ is   semi-Fredholm   and $T_{N}^{n}$ is zeroloid. Thus    $\mbox{ind}(T^{n})=\mbox{ind}(T_{M}^{n})=n.\mbox{ind}(T_{M})=n.\mbox{ind}(T).$\\
(ii)   Since   $T \in L(X)$ and $S \in L(Y)$  are    semi-$g_{z}$-Fredholm, then  there exist $(M_{1}, N_{1}) \in \mbox{Red}(T)$ and  $(M_{2}, N_{2}) \in Red(S)$ such that $T_{M_{1}}$ and $T_{M_{2}}$  are  semi-Fredholm,  $T_{N_{1}}$ and $T_{N_{2}}$  are  zeroloid. Hence $T_{M_{1}\oplus M_{2}}$  is   semi-Fredholm and $T_{N_{1}\oplus N_{2}}$ is zeroloid.   It is easily seen that $(M_{1}\oplus M_{2}, N_{1}\oplus N_{2})\in Red(T\oplus S),$ and hence   $\mbox{ind} (T\oplus S)=\mbox{ind} ((T\oplus S)_{M_{1}\oplus M_{2}})=\mbox{ind} (T_{M_{1}})+ \mbox{ind} (S_{M_{2}})=\mbox{ind} (T)+ \mbox{ind} (S).$
\end{proof}
Denote  by  $\sigma_{ug_{z}f}(T),$ $\sigma_{lg_{z}f}(T),$ $\sigma_{sg_{z}f}(T),$ $\sigma_{g_{z}f}(T),$ $\sigma_{ug_{z}w}(T),$ $\sigma_{lg_{z}w}(T),$ $\sigma_{sg_{z}w}(T)$ and  $\sigma_{g_{z}w}(T)$  respectively, the upper   semi-$g_{z}$-Fredholm spectrum, the  lower  semi-$g_{z}$-Fredholm spectrum, the    semi-$g_{z}$-Fredholm, the      $g_{z}$-Fredholm spectrum spectrum,  the upper   semi-$g_{z}$-Weyl spectrum, the  lower  semi-$g_{z}$-Weyl spectrum, the    semi-$g_{z}$-Weyl spectrum and the     $g_{z}$-Weyl spectrum   of a given $T \in L(X).$
\begin{cor}\label{cor0} For every  $T \in L(X),$ we have   $\sigma_{g_{z}f}(T)=\sigma_{ug_{z}f}(T)\cup\sigma_{lg_{z}f}(T)$ and $\sigma_{g_{z}w}(T)=\sigma_{ug_{z}w}(T)\cup\sigma_{lg_{z}w}(T).$
\end{cor}
\begin{prop}
Let $T \in L(X)$ be  a semi-B-Fredholm operator which  is    semi-$g_{z}$-Fredholm.   Then $T$  is quasi  semi-B-Fredholm  and   its  index  as a semi-B-Fredholm coincides with its index as a semi-$g_{z}$-Fredholm.
\end{prop}
\begin{proof} Let $(M, N)\in \mbox{Red}(T)$  such that  $T_{M}$ is  semi-Fredholm    and $T_{N}$ is zeroloid.   Hence $T_{N}$ is  Drazin invertible, since $T$ is semi-B-Fredholm.  So there exists $(A, B)\in Red(T_{N})$  such that   $T_{A}$ is   invertible     and $T_{B}$ is  nilpotent. It is easy to get that   $M \oplus A$ is a closed subspace,  so that  $T_{M \oplus A}$ is  semi-Fredholm. Consequently,   $T=T_{M \oplus A}\oplus T_{B}$ is quasi  semi-B-Fredholm.    Furthermore,  the punctured neighborhood theorem for semi-Fredholm operators implies  that $\mbox{ind}(T_{M})=\mbox{ind}(T_{[m_{T}]}).$
\end{proof}
From  \cite[Theorem 7]{mullerB} and the previous proposition we obtain the following corollary.
\begin{cor}
 Every B-Fredholm   operator $T \in L(X)$  is  $g_{z}$-Fredholm and its usual  index  coincides with its index as a $g_{z}$-Fredholm.
\end{cor}

\begin{prop}\label{propp1}  If  $T \in L(X)$ is  a  semi-$g_{z}$-Fredholm operator  then $T^{*}$ is   semi-$g_{z}$-Fredholm,  $\tilde{\alpha}(T)=\tilde{\beta}(T^{*}),$ $\tilde{\beta}(T)=\tilde{\alpha}(T^{*})$  and  $\mbox{ind}(T)=-\mbox{ind}(T^{*}).$
\end{prop}
\begin{proof} See  Lemma \ref{lem1.8}.
\end{proof}
Our next definition gives  a new  class of operators that extends the class of semi-Browder operators.
\begin{dfn} We say that $T \in L(X)$ is  an upper semi-$g_{z}$-Browder (resp., a lower  semi-$g_{z}$-Browder, $g_{z}$-Browder) if $T$ is a direct sum  of an upper  semi-Browder (resp., a lower  semi-Browder, Browder) operator and a zeroloid operator.
\end{dfn}

\begin{prop}\label{prop2} Let $T \in L(X).$   The following statements are equivalent.\\
(i) $T$ is  an upper semi-$g_{z}$-Browder {\rm[}resp.,   $T$ is  a lower  semi-$g_{z}$-Browder,  $T$ is $g_{z}$-Browder{\rm]};\\
(ii) $T$ is  an upper $g_{z}$-Weyl  and $T$  has the SVEP at 0  {\rm[}resp.,   $T$ is a  lower semi-$g_{z}$-Weyl and $T^{*}$  has the SVEP at 0, $T$ is   $g_{z}$-Weyl and $T$ or  $T^{*}$ has the SVEP at 0{\rm]};\\
(iii)  $T$ is an upper semi-$g_{z}$-Fredholm  and $T$ has the SVEP at 0 {\rm[}resp.,  $T$ is a  lower  semi-$g_{z}$-Fredholm  and $T^{*}$ has the SVEP at 0,    $T$ is  $g_{z}$-Fredholm and $T\oplus T^{*}$ has the SVEP at 0{\rm]}.
\end{prop}
\begin{proof}
 (i) $\Longleftrightarrow$ (ii) Suppose that $T$ is $g_{z}$-Browder, then there exists $(M,N)\in g_{z}KD(T)$ such that  $T_{M}$ is Browder. So $T_{M},$ $(T_{M})^{*},$ $T_{N}$ and $(T_{N})^{*}$  have the SVEP at 0. Thus $T$ and $T^{*}$ have the SVEP at $0.$ Conversely,   if $T$ is $g_{z}$-Weyl and $T$ or  $T^{*}$ has the SVEP at 0 then  there exists $(M, N) \in g_{z}KD(T)$ such that  $T_{M}$ is Weyl  and $T_{M}$  or $(T_{M})^{*}$ has  the SVEP at 0.  So   $\mbox{max}\{\tilde{\alpha}(T), \tilde{\beta}(T)\}<\infty$ and  $\mbox{min}\{\tilde{p}(T), \tilde{q}(T)\}<\infty.$ This implies  from  \cite[Lemma 1.22]{aiena} that     $\mbox{max}\{\tilde{p}(T), \tilde{q}(T)\}<\infty$ and  then    $T_{M}$ is Browder. Therefore   $T$ is $g_{z}$-Browder. The other  equivalence cases go similarly.\\
 (i) $\Longleftrightarrow$ (iii)  Suppose that  $T$ is $g_{z}$-Fredholm and $T\oplus T^{*}$ has the SVEP at $0.$ Let $(M, N) \in g_{z}KD(T)$  such that  $T_{M}$ is a Fredholm  operator  and $T_{N}$ is  zeroloid.   Hence  $T_{M}\oplus (T_{M})^{*}$  has the SVEP at 0,  which implies from  the implications  (A) and (B) defined in the introduction part  that $T_{M}$ is  Browder and then  $T$ is a $g_{z}$-Browder. The converse is clear and the other equivalence cases  go similarly.
\end{proof}
The proof of  the following results are obvious and is left to the reader.
\begin{prop}  If $T\in L(X)$ is semi-$g_{z}$-Fredholm, then there exists $\epsilon>0$ such that $B_{0}:=B(0, \epsilon)\setminus \{0\} \subset (\sigma_{spbf}(T))^C$ and $\mbox{ind}(T)=\mbox{ind}(T- \lambda I)$ for every $\lambda \in B_{0}.$
\end{prop}
\begin{cor} For every $T \in L(X),$ the following assertions hold.\\
(i)  $\sigma_{ug_{z}f}(T),$ $\sigma_{lg_{z}f}(T),$ $\sigma_{sg_{z}f}(T),$ $\sigma_{g_{z}f}(T),$ $\sigma_{ug_{z}w}(T),$ $\sigma_{lg_{z}w}(T),$ $\sigma_{sg_{z}w}(T)$ and  $\sigma_{g_{z}w}(T)$  are  compact.\\
(ii) If $\Omega$ is a component of $(\sigma_{ug_{z}f}(T))^C$ or  $(\sigma_{lg_{z}f}(T))^C,$ then the index $\mbox{ind}(T -\lambda I)$ is constant as $\lambda$ ranges over $\Omega.$
\end{cor}
\begin{cor}\label{coracc}
  Let $T \in L(X).$ The following statements are equivalent.\\
(i)  $T$ is  semi-$g_{z}$-Weyl {\rm[}resp.,   upper semi-$g_{z}$-Weyl, lower semi-$g_{z}$-Weyl, $g_{z}$-Weyl{\rm]};\\
(ii)  $T$  is $g_{z}$-Kato and  $0 \notin \mbox{acc}\,\sigma_{spbw}(T)$  {\rm[}resp.,   $T$  is $g_{z}$-Kato and $0 \notin \mbox{acc}\,\sigma_{upbw}(T),$ $T$  is $g_{z}$-Kato and $0 \notin \mbox{acc}\,\sigma_{lpbw}(T),$  $T$  is $g_{z}$-Kato and $0 \notin \mbox{acc}\,\sigma_{pbw}(T)${\rm]}.
\end{cor}

\section{$g_{z}$-invertible operators}
Recall \cite{aiena} that  $T\in L(X)$ is Drazin invertible if and only if there exists an operator  $S \in L(X)$ which commutes with $T$ with $STS=S$ and $T^{n}ST=T^{n}$ for some integer $n\in \N.$ The index of a Drazin invertible operator $T$  is defined by   $i(T):=\mbox{min}\{n\in \N : \exists   S \in L(X) \text{ such that } ST=TS, STS=S \text{ and } T^{n}ST=T^{n}\}.$
\begin{prop}\label{propi} Let $T \in L(X).$ If $p(T)<\infty$ (resp., $q(T)<\infty$)  then $p(T)=\mbox{dis}(T)$ (resp., $q(T)=\mbox{dis}(T)$). Moreover,  if $T$ is Drazin invertible  then $i(T)=\mbox{dis}(T).$
\end{prop}
\begin{proof} Suppose that $p(T)<\infty,$ then   $\mathcal{N}(T_{[n]})=\{0\}$ for every integer $n\geq p(T).$    This  implies   that   $\mathcal{N}(T_{[d]})=\{0\};$ where $d:=\mbox{dis}(T).$    Thus  $p(T)\leq d,$ and since we always  have  $d\leq \mbox{min}\{p(T),q(T)\}$  then  $p(T)=d.$ If $q(T)<\infty$ then $X=\R(T)+\mathcal{N}(T^{n})$ for every integer $n\geq q(T).$  Since   $\R(T)+\mathcal{N}(T^{d})=\R(T)+\mathcal{N}(T^{m})$  for every integer $m\geq d,$ then    $X=\R(T)+\mathcal{N}(T^{d}).$   Hence $T_{[d]}$ is surjective and  consequently  $q(T)=d.$ If in addition $T$ is Drazin invertible, then the  proof of the equality desired   is an immediate consequence of    \cite[Theorem 1.134]{aiena}.
\end{proof}
\begin{dfn}\label{dfn.i3} We say that $T$ is quasi left Drazin invertible (resp., quasi right Drazin invertible)  if there exists $(M,N) \in KD(T)$ such that $T_{M}$ is bounded below (resp., surjective).
\end{dfn}
\begin{prop}\label{prop.i4} Let $T \in L(X).$  The following hold.\\
(i)  $T$ is Drazin invertible if and only if $T$ is  quasi left and  quasi right Drazin invertible.\\
(ii) If $T$ is quasi left Drazin invertible then $T$ is left Drazin invertible.\\
(iii) If $T$ is quasi right Drazin invertible then $T$ is right Drazin invertible.\\
 Furthermore,   the converses of (ii) and (iii) are   true in the case  of  Hilbert space.
\end{prop}
\begin{proof} (i) Assume  that $T$ is Drazin invertible, then $n:=p(T)=q(T)<\infty.$ It is well known that $(\R(T^{n}),\mathcal{N}(T^{n}))\in \mbox{Red}(T),$  $T_{\R(T^{n})}$ is invertible and $T_{\mathcal{N}(T^{n})}$ is nilpotent. So $T$ is  quasi left  and quasi right Drazin invertible. Conversely, if $T$ is  quasi left and  quasi right Drazin invertible  then $\tilde{\alpha}(T)=\tilde{\beta}(T)=0.$ Therefore $\alpha(T_{M})=\tilde{\alpha}(T)=\tilde{\beta}(T)=\beta(T_{M})=0,$ for every  $(M,N) \in KD(T).$  Thus  $T$ is Drazin invertible.\\
(ii) Let $(M,N) \in \mbox{Red}(T)$ such that $T_{M}$ is bounded below and $T_{N}$ is nilpotent of degree $d.$ As a bounded below operator is semi-regular,   we deduce from \cite[Theorem 2.21]{aznay-ouahab-zariouh2}  that $d=\mbox{dis}(T).$ Clearly,  $\R(T^{n})$ is closed and $T_{[n]}=(T_{M})_{[n]}$ is bounded below for every integer $n \geq d.$  Hence    $T$ is left Drazin invertible. Conversely,  assume that  $T$ is left Drazin invertible Hilbert space operator. Then  $T$ is upper semi-B-Fredholm,  which entails  from \cite[Theorem 2.6]{berkani-sarih} and \cite[Corollary 3.7]{aznay-ouahab-zariouh2} that there exists $(M,N) \in KD(T)$ such that  $T_{M}$ is upper  semi-Browder. Using  \cite[Lemma 2.17]{aznay-ouahab-zariouh}, we conclude  that   $T_{M}$ is bounded below and then $T$ is quasi left Drazin invertible.\\
(iii) Goes similarly with   (ii).
\end{proof}
\begin{prop}\label{prop.i5}  $T \in L(X)$ is an upper semi-Browder  {\rm[}resp., a lower semi-Browder{\rm]} if and only if  $T$ is a quasi left Drazin invertible  {\rm[}resp.,  a quasi right Drazin invertible{\rm]} and    $\mbox{dim}\,N<\infty$ for every  (or for some)  $(M,N) \in KD(T).$
\end{prop}
\begin{proof} Suppose that $T$ is an upper semi-Browder, then $T$ is upper semi-Fredholm. From \cite[Corollary 3.7]{aznay-ouahab-zariouh2}, there exists $(M,N) \in KD(T)$  such that  $T_{M}$ is  upper semi-Browder.  It follows  from \cite[Lemma 2.17]{aznay-ouahab-zariouh} that $T_{M}$  is bounded below. Let $(A,B) \in KD(T)$ be arbitrary. Since a nilpotent operator $S$ acting on a Banach space $Y$ is semi-Fredholm if and only if $\mbox{dim}\,Y<\infty,$ then $\mbox{dim}\,B<\infty.$  The converse is obvious. The other case goes similarly.
\end{proof}
\begin{dfn}\label{dfn.i1}   $T \in L(X)$ is said to be left $g_{z}$-invertible (resp., right $g_{z}$-invertible)  operator    if there exists $(M, N) \in g_{z}KD(T)$ such that  $T_{M}$ is bounded below (resp.,   surjective). $T$ is called $g_{z}$-invertible if it is left  and right $g_{z}$-invertible.
\end{dfn}

\begin{rema}\label{rema.i2}
(i) It is clear    that $T$ is  $g_{z}$-invertible  if and only if there exists $(M, N) \in g_{z}KD(T)$ such that  $T_{M}$ is invertible.\\
(ii) Every    generalized Drazin-meromorphic  invertible operator is  $g_{z}$-invertible.
\end{rema}
We prove in the following results  that the class of $g_{z}$-invertible operators   preserves some  properties  of the  of Drazin invertibility    \cite{drazin,koliha}.
\begin{thm}\label{thm.i6} Let $T\in L(X).$ The following statements are equivalent.\\
(i) $T$ is $g_{z}$-invertible;\\
(ii) $T$ is $g_{z}$-Browder;\\
(iii) There exists $(M,N) \in g_{z}KD(T)$ such that  $T_{M}$ is Drazin invertible;\\
(iv) There exists  a Drazin invertible operator $S \in L(X)$  such that $T S = ST,$ $STS = S$ and  $T^{2}S - T$ is zeroloid.   A  such $S$ is called a  $g_{z}$-inverse of $T;$\\
(v) There exists a bounded projection $P$ on $X$ which   commutes with $T,$  $T + P$ is generalized Drazin invertible  and $TP$  is zeroloid;  \\
(vi)  There exists a bounded projection $P$ on $X$    commuting with $T$ such that   there exist $U,V\in L(X)$  which satisfy   $P=TU=VT$ and   $T(I-P)$ is zeroloid;\\
(vii) $T$ is $g_{z}$-Kato and $\tilde{p}(T)=\tilde{q}(T)<\infty.$
\end{thm}
\begin{proof}  The  equivalences (i) $\Longleftrightarrow$ (ii)  and (i) $\Longleftrightarrow$ (iii) are  immediate  consequences  of  Propositions  \ref{prop.i4} and \ref{prop.i5}.\\
(i) $\Longleftrightarrow$ (iv)  Assume that $T$ is $g_{z}$-invertible and let  $(M, N) \in g_{z}KD(T)$ such that    $T_{M}$ is  invertible.   The operator  $S=(T_{M})^{-1}\oplus 0_{N}$  is Drazin invertible. Moreover,   $TS = ST=I_{M}\oplus 0_{N},$ $STS = S$ and $T^{2}S - T=0_{M}\oplus (-T_{N}).$ As $T_{N}$ is zeroloid, then  $T^{2}S - T$ is also zeroloid. Conversely, suppose that there exists a Drazin invertible  operator $S$ such that $T S = ST,$ $STS = S$ and  $T^{2}S - T$ is zeroloid.  So   $TS$ is a projection.  If we take  $M=\R(TS)$ and  $N=\mathcal{N}(TS),$ then $(M,N)\in \mbox{Red}(T)\cap Red(S).$  Moreover, if $x\in \mathcal{N}(T_{M})$ then $x=TSy$ and $Tx=0,$ thus $x=(TS)^2y=STx=0.$  Moreover,   $\R(T_{M})=M$ and then   $T_{M}$ is  invertible.  Let us to show that $S=(T_{M})^{-1}\oplus 0_{N}.$ We have   $S_{N}=0_{N},$ since   $S=STS.$    Let $x=TSy \in M,$  as $Sy=STSy \in M$ then $Sx=Sy=(T_{M})^{-1}T_{M}Sy=(T_{M})^{-1}x.$ Hence $S=(T_{M})^{-1}\oplus 0$  and  $T^{2}S - T=0\oplus T_{N}.$ This  implies  that  $T_{N}$ is zeroloid and then  $T$  is  $g_{z}$-invertible.\\
(i) $\Longleftrightarrow$ (v) Suppose that there exists a bounded projection $P$ on $X$ which   commutes with $T,$  $T + P$ is generalized Drazin invertible  and $TP$  is zeroloid.  Then $(A,B):=(\mathcal{N}(P),\R(P))\in \mbox{Red}(T),$  $T_{A}=(T+P)_{A}$ is generalized Drazin invertible   and $T_{B}=(TP)_{B}$ is zeroloid.  Thus  there exists $(C,D) \in Red(T_{A})$ such that $T_{C}$ is invertible and $T_{D}$ is quasi-nilpotent. Hence  $(C,D\oplus B)\in g_{z}KD(T)$  and then  $T$ is $g_{z}$-invertible. Conversely,   let $(M,N) \in g_{z}KD(T)$ such that $T_{M}$ is invertible. Clearly,   $P:=0_{M} \oplus I_{N}$  is a projection and $TP=PT.$ Furthermore, $TP= 0_{M} \oplus T_{N}$  is zeroloid and $T+P=T_{M}\oplus (T+I)_{N}$ is generalized Drazin invertible, since  $-1\notin \mbox{acc}\,\sigma(T_{N})=\sigma_{gd}(T_{N}).$\\
(vi)   $\Longrightarrow$  (i)  Suppose that there exists a bounded projection $P$ on $X$    commuting with $T$ such that   there exist $U,V\in L(X)$  which satisfy   $P=TU=VT$ and   $T(I-P)$ is zeroloid.   Then   $I_{M}\oplus 0_{N}=T_{M}U_{M}\oplus T_{N}U_{N}=V_{M}T_{M}\oplus V_{N}T_{N};$  where   $(M,N):=(\R(P),\mathcal{N}(P))\in \mbox{Red}(T),$ and thus    $T_{M}U_{M}=V_{M}T_{M}=I_{M}$ and $T_{N}U_{N}=V_{N}T_{N}=0_{N}.$ Hence $T_{M}$ is invertible. Moreover,  $T_{N}$ is zeroloid, since  $T(I-P)=0\oplus T_{N}$  is zeroloid. Consequently, $T$ is $g_{z}$-invertible.\\
(iv) $\Longrightarrow$ (vi)   and  (i) $\Longleftrightarrow$ (vii) are clear.
\end{proof}
The two following theorems are analogous to the preceding one and  we give them without proof.
\begin{thm}\label{thm.i7} Let $T\in L(X).$ The following statements are equivalent.\\
(i) $T$ is left $g_{z}$-invertible;\\
(ii) $T$ is upper semi-$g_{z}$-Browder;\\
(iii) There exists $(M,N) \in g_{z}KD(T)$ such that $T_{M}$ is quasi right  Drazin invertible;\\
(v) There exists a bounded projection $P$ on $X$ which commutes with $T,$ $T + P$ is left generalized Drazin invertible  and $TP$  is zeroloid;\\
(vi) $T$ is $g_{z}$-Kato and $\tilde{p}(T)=0;$\\
(vii) $T$ is $g_{z}$-Kato and $0\not\in\mbox{acc}\,\sigma_{lgd}(T).$
\end{thm}
\begin{thm}\label{thm.i8} Let $T\in L(X).$ The following statements are equivalent.\\
(i) $T$ is right $g_{z}$-invertible;\\
(ii) $T$ is lower semi-$g_{z}$-Browder;\\
(iii) There exists $(M,N) \in g_{z}KD(T)$ such that  $T_{M}$ is quasi right Drazin invertible;\\
(v) There exists a bounded projection $P$ on $X$ which commutes with $T,$ $T + P$ is right generalized Drazin invertible and $TP$  is zeroloid;\\
(vi) $T$ is $g_{z}$-Kato and $\tilde{q}(T)=0;$\\
(vii) $T$ is $g_{z}$-Kato and $0\not\in\mbox{acc}\,\sigma_{rgd}(T).$
\end{thm}
\begin{cor}\label{cor.i9}  If   $T \in  L(X)$ is  $g_{z}$-invertible  and  $S$ is a   $g_{z}$-inverse of $T,$ then     $TST$ is the  Drazin  inverse of $S$   and $p(S) = q(S) = \mbox{dis}(S)\leq 1.$
\end{cor}
\begin{proof} Obvious.
\end{proof}
Hereafter, $\sigma_{lg_{z}d}(T),$ $\sigma_{rg_{z}d}(T)$ and  $\sigma_{g_{z}d}(T)$  are  respectively, the left    $g_{z}$-invertible spectrum, the right    $g_{z}$-invertible spectrum and the     $g_{z}$-invertible spectrum   of a given $T \in L(X).$
 \begin{thm}\label{thm.i10}
For every  $T\in L(X)$ we have   $\sigma_{g_{z}d}(T)=\mbox{acc}\,(\mbox{acc}\,\sigma(T)).$
\end{thm}
\begin{proof}
Let $\mu \notin \mbox{acc}\,(\mbox{acc}\,\sigma(T)).$ Without loss of generality we assume that $\mu=0$ [note that $\mbox{acc}\,\mbox{acc}\,\sigma(T- \alpha I)=\mbox{acc}\,(\mbox{acc}\,\sigma(T))-\alpha,$ for every complex scalar $\alpha$]. If $0 \notin \mbox{acc}\,\sigma(T),$ then $T$ is generalized Drazin invertible and it   is in particular  $g_{z}$-invertible.\\
If $0 \in \mbox{acc}\,\sigma(T)$ then $0\in \mbox{acc}\,(\mbox{iso}\, \sigma(T)).$ We distinguish two  cases:\\
{\bf Case 1:}  $\mbox{acc}\,(\mbox{iso}\, \sigma(T))\neq\{0\}.$  It follows that    $\epsilon:=\underset{\lambda\in acc\,(iso\, \sigma(T))\setminus\{0\}}{\mbox{inf}} |\lambda|>0.$ The sets $F_{2}:= D(0, \frac{\epsilon}{2}) \cap \overline{\mbox{iso}\, \sigma(T)}$ and   $F_{1}:=((\mbox{acc}\,\sigma(T))\setminus\{0\}) \cup (\overline{\mbox{iso}\, \sigma(T)}\setminus F_{2})$   are closed  and disjoint. Indeed, $F_{1}\cap F_{2}=F_{2}\cap\left[(\mbox{acc}\,\sigma(T))\setminus\{0\}\right]\subset  \left[\mbox{acc}\, (\mbox{iso}\, \sigma(T))\setminus\{0\}\right] \cap D(0, \frac{\epsilon}{2}) =\emptyset.$  As $0 \notin \mbox{acc}\,(\mbox{acc}\,\sigma(T))$ then $(\mbox{acc}\,\sigma(T))\setminus\{0\}$ is closed.  Let us  show that  $C:=(\overline{\mbox{iso}\, \sigma(T)}\setminus F_{2})$ is closed.   Suppose  that   $\lambda \in \mbox{acc}\,{C}$ (the case of  $\mbox{acc}\,{C}=\emptyset$ is obvious), then $\lambda \in \overline{\mbox{iso}\, \sigma(T)}.$    Let   $(\lambda_{n})_{n}\subset C$ be   a non stationary   sequence which  converges to $\lambda,$ it follows that    $\lambda \neq 0.$ We have $\lambda \not\in F_{2}.$ Otherwise, $\lambda  \in D(0, \frac{\epsilon}{2})$ and then  $\lambda \notin \mbox{acc}\,(\mbox{iso}\, \sigma(T)),$  so that  $\lambda \in \mbox{iso}\, \sigma(T)$ and  this is a contradiction. Therefore  $C$ is closed  and then  $F_{1}$ is closed.   As $\sigma(T)=F_{1}\cup F_{2},$ then    there exists $(M, N)\in  \mbox{Red}(T)$  such that $\sigma(T_{M})=F_{1}$ and $\sigma(T_{N})=F_{2}.$    So  $T_{M}$ is invertible  and   $0  \in  \mbox{acc}\, \sigma(T_{N}).$  Let $v  \in F_{2}$  then $v \notin \mbox{acc}\,\sigma(T_{N})\setminus\{0\},$ since $F_{1}\cap F_{2}=F_{2}\cap(\mbox{acc}\,\sigma(T)\setminus\{0\})=\emptyset.$  Hence $\mbox{acc}\,\sigma(T_{N})= \{0\}$  and $T$ is $g_{z}$-invertible.\\
{\bf Case 2:}  $\mbox{acc}\,(\mbox{iso}\, \sigma(T))=\{0\}.$  Then      $F_{2}:= D(0, 1) \cap \overline{\mbox{iso}\, \sigma(T)}$ and $F_{1}:=((\mbox{acc}\,\sigma(T))\setminus\{0\}) \cup (\overline{\mbox{iso}\, \sigma(T)}\setminus F_{2})$ are closed disjoint subsets and give  the desired result. For this, let $\lambda \in \overline{C};$ where $C:=\overline{\mbox{iso}\, \sigma(T)}\setminus F_{2}.$ Then there exists a sequence $(\lambda_{n})\subset C$ which  converges to $\lambda.$ As  $\mbox{acc}\,(\mbox{iso}\, \sigma(T))=\{0\}$ and $\lambda (\neq 0) \in \overline{\mbox{iso}\, \sigma(T)}$ then $\lambda \in \mbox{iso}\, \sigma(T).$ Therefore   $(\lambda_{n})_{n}$ is stationary and so $\lambda \in C.$ Thus $F_{1}$ is closed  and hence    there exists $(M, N)\in  \mbox{Red}(T)$  such that $\sigma(T_{M})=F_{1}$ and $\sigma(T_{N})=F_{2}.$    Conclusion,   $T$ is $g_{z}$-invertible.\\
Conversely, if  $T$ is $g_{z}$-invertible then $T=T_{1}\oplus T_{2};$ where $T_{1}$ is invertible and $T_{2}$ is zeroloid. And then    there exists $\epsilon>0$ such that $B(0, \epsilon)\setminus \{0\} \subset (\sigma(T_{1}))^C\cap(\mbox{acc}\,\sigma(T_{2}))^C\subset (\mbox{acc}\,\sigma(T))^C.$ Thus   $0 \notin \mbox{acc}\,(\mbox{acc}\,\sigma(T))$ and the proof is complete.
\end{proof}
From  the previous theorem and some well known results in perturbation theory, we obtain the following corollary.
\begin{cor}\label{cor.i11}  Let $T \in L(X).$   The following statements  hold.\\
(i)  $\sigma_{lg_{z}d}(T),$ $\sigma_{rg_{z}f}(T)$  and  $\sigma_{g_{z}d}(T)$  are  compact.\\
(ii) $\sigma_{g_{z}d}(T)=\sigma_{g_{z}d}(T^{*}).$\\
(iii) If $S$ is  a  Banach space operator, then $T\oplus S$ is     $g_{z}$-invertible  if and only if $T$ and $S$ are    $g_{z}$-invertible.\\
(iv) $T$   is  $g_{z}$-invertible if and only if   $T^{n}$ is  $g_{z}$-invertible for some (equivalently for every) integer $n\geq 1.$\\
(v) If $Q \in \mbox{comm}(T)$ is  quasi-nilpotent,  then  $\sigma_{g_{z}d}(T)= \sigma_{g_{z}d}(T+Q).$\\
(vi) If $F \in \mathcal{F}_{0}(X)\cap\mbox{comm}(T),$  then  $\sigma_{g_{z}d}(T)= \sigma_{g_{z}d}(T+F);$ where  $\mathcal{F}_{0}(X)$ is  the set of   all power finite rank operators.
\end{cor}
\begin{ex}\label{exs} (i) Let $T \in L(X)$ be the    operator such that $\sigma(T)=\sigma_{d}(T)=\overline{\{\frac{1}{n}\}}.$ Then  $T$ is  $g_{z}$-invertible which  is not  generalized Drazin-meromorphic invertible, since $0\in \mbox{acc}\,\sigma_{d}(T)$ (see \cite[Theorem 5]{rwassa2}). Note also that  $T$ is not  generalized Kato-meromorphic. For this, suppose  that $T$ is generalized Kato-meromorphic then $\tilde{\alpha}(T)=\tilde{\beta}(T)=0$, since $T$ is $g_{z}$-invertible.  Hence $T$ is   generalized Drazin-meromorphic invertible and this  is a contradiction.
\end{ex}
\begin{prop}\label{prop.i12} Let $T \in L(X).$  The following statements are equivalent.\\
(i) $0\in \mbox{iso}\,(\mbox{acc}\,\sigma(T))$ (i.e. $T$ is  $g_{z}$-invertible   which  is not  generalized Drazin invertible);\\
(ii) $T = T_{1} \oplus T_{2};$  where $T_{1}$ is invertible and  $\mbox{acc}\,\sigma(T_{2})=\{0\};$\\
(iii) $T$ is  $g_{z}$-Kato  and there exists a non stationary  sequence  of isolated points of  $\sigma(T)$ that converges  to 0.
\end{prop}
\begin{proof}
(i) $\Longrightarrow$ (ii) Follows directly from the  proof of Theorem \ref{thm.i10}. Note   that $\mbox{acc}\,\sigma(T_{N})=\{0\}$ for every $(M,N) \in g_{z}KD(T).$\\
(ii) $\Longrightarrow$ (iii) As  $T = T_{1} \oplus T_{2},$   $T_{1}$ is invertible and  $\mbox{acc}\,\sigma(T_{2})=\{0\},$ then $0 \in \mbox{iso}\,(\mbox{acc}\,\sigma(T))$  and    there exists a non stationary  sequence   $(\lambda_{n})_{n} \subset \mbox{iso}\,\sigma(T_{2})$  which   converges to 0.  Thus $T$ is $g_{z}$-invertible and  there exists $N \in \mathbb{N}$ such that  $\lambda_{n} \in \sigma(T)\setminus\mbox{acc}\,\sigma(T)=\mbox{iso}\,\sigma(T)$ for all $n \geq N.$\\
(iii) $\Longrightarrow$ (i) Assume  that  $T = T_{1} \oplus T_{2};$   $T_{1}$ is semi-regular,  $T_{2}$ is zeroloid and  there exists  a non stationary  sequence $(\lambda_{n})_{n}$    of isolated point of  $\sigma(T)$ that converges to 0.  Hence  $0 \in \mbox{acc}\,\sigma(T)$ and   $T\oplus T^{*}$ has the SVEP at 0.  This entails    that $T$  is  $g_{z}$-invertible and  then  $0\in \mbox{iso}\,(\mbox{acc}\,\sigma(T)).$
\end{proof}
Recall that $\sigma\subset \sigma(T)$ is called  a spectral set (called also isolated part) of $T$ if  $\sigma$ and $\sigma(T)\setminus\sigma$ are closed, see \cite{gohbergbook}.\\
Let $T$ be  a $g_{z}$-invertible operator which is not generalized Drazin invertible.  From  Proposition \ref{prop.i12}, we conclude  that there exists a non-zero  strictly  decreasing sequence $(\lambda_{n})_{n}\subset \mbox{iso}\,\sigma(T)$  which    converges to 0 such that    $\sigma:=\overline{\{\lambda_{n} : n\in \N\}}$ is a spectral set  of $T.$    If    $P_{\sigma}$ is   the spectral projection associated to  $\sigma,$ then  $(M_{\sigma},N_{\sigma}):=(\mathcal{N}(P_{\sigma}),\R(P_{\sigma}))\in g_{z}KD(T),$ $\sigma(T_{N_{\sigma}})=\sigma$ and $\sigma(T_{M_{\sigma}})=\sigma(T)\setminus\sigma.$ Thus   $T+rP_{\sigma}=T_{M_{\sigma}}\oplus (T+rI)_{N_{\sigma}}$ is invertible for every $|r|> |\lambda_{0}|$ and then   the operator  $T^{D}_{\sigma}:=(T+rP_{\sigma})^{-1}(I-P_{\sigma})=(T_{M_{\sigma}})^{-1}\oplus 0_{N_{\sigma}}$ is  a $g_{z}$-inverse of $T$ and   depend only of $\sigma.$  Note  that  $P_{\sigma}=I-TT^{D}_{\sigma}\in \mbox{comm}^{2}(T):=\{S \in \mbox{comm}(L):  L \in \mbox{comm}(T)\},$ so    that $(M_{\sigma},N_{\sigma}) \in \mbox{Red}(S)$ for every  operator $S\in \mbox{comm}(T)$ and $T^{D}_{\sigma}\in   \mbox{comm}^{2}(T).$  Note also that $T+P_{\sigma}$ is generalized Drazin invertible and $TP_{\sigma}$ is zeroloid.
\begin{lem}\label{lemomom} Let $T \in L(X)$ be a $g_{z}$-invertible operator and  $(M,N)\in g_{z}KD(T)$ such that $T_{M}$ invertible and  $\sigma(T_{M})\cap\sigma(T_{N})=\emptyset.$ Then   $\sigma(T_{N})\setminus\{0\}\subset \mbox{iso}\,\sigma(T)$ and  for every $S\in \mbox{comm}(T),$  $(M,N)\in \mbox{Red}(S).$
\end{lem}
\begin{proof}   If $T$ is generalized Drazin invertible, then   $0\notin \mbox{acc}\,\sigma(T)$ and so $\mbox{acc}\,\sigma(T_{N})=\emptyset,$ hence $\sigma(T_{N})$ is a finite set of isolated points of $\sigma(T).$    Let $P_{\sigma}$  the spectral projection associated to $\sigma=\sigma(T_{N}),$  \cite[Proposition 2.4]{gohbergbook} and the fact that $P_{\sigma} \in \mbox{comm}^{2}(T)$  imply that  $(M,N)=(\mathcal{N}(P_{\sigma}),\R(P_{\sigma})) \in \mbox{Red}(S)$  for every $S\in \mbox{comm}(T).$   If   $T$ isn't generalized Drazin invertible, then there exists a non-zero  strictly  decreasing sequence $(\lambda_{n})_{n}$ which    converges to 0 such that $\sigma(T_{N})=\overline{\{\lambda_{n} : n\in \N\}}.$ On the other hand,     there exists $\epsilon>0$ such that $B(0,\epsilon)\cap\sigma(T_{M})=\emptyset$ and hence $\sigma(T_{N})\setminus\{0\}\subset \mbox{iso}\,\sigma(T).$   Let  $P$ be   the spectral projection associated to the spectral set $\sigma(T_{N}),$ then $(M,N)=(\mathcal{N}(P),\R(P))$ and so   $(M,N)\in \mbox{Red}(S)$ for every $S\in \mbox{comm}(T).$
\end{proof}
\begin{rema}  It is  not difficult  to see that the condition,  $\exists (M,N)\in \mbox{Red}(S)$ such that $T_{M}$ is invertible   for every $S\in \mbox{comm}(T)$  is equivalent to   $\exists L\in \mbox{comm}^{2}(T)$ such that $L=L^{2}T.$
\end{rema}
\begin{thm} Let $T\in L(X).$ Then the following statements are equivalent.
\begin{enumerate}[nolistsep]
\item[(i)] $T$  is $g_{z}$-invertible;
\item[(ii)] $0\notin \mbox{acc}(\mbox{acc}\,\sigma(T))$;
\item[(iii)]  there exists $(M,N)\in g_{z}KD(T)$ such that $T_{M}$ invertible and  $\sigma(T_{M})\cap\sigma(T_{N})=\emptyset;$
\item[(iv)] there exists  a spectral set $\sigma$   of $T$ such that  $0\notin \sigma(T)\setminus\sigma$    and  $\sigma\setminus\{0\}\subset \mbox{iso}\,\sigma(T).$
\item[(v)] There exists a bounded projection   $P\in \mbox{comm}^{2}(T)$    such that  $T+P$ is generalized Drazin invertible    and  $TP$ is zeroloid.
\end{enumerate}
\end{thm}
\begin{proof} The equivalence (i) $\Longleftrightarrow$ (ii) is proved in Theorem \ref{thm.i10} and the equivalences (i) $\Longleftrightarrow$ (iii) and (i) $\Longleftrightarrow$ (v) are    proved in the paragraph preceding  Lemma \ref{lemomom} (the  case of $T$ is generalized Drazin invertible  is clear). The proof of the   equivalence (i) $\Longleftrightarrow$ (iv) is based on the   same techniques used    in the one  of Lemma  \ref{lemomom} and we left it to the reader.
\end{proof}
\begin{prop}\label{prop.unicinverse} For every  $T \in L(X)$ $g_{z}$-invertible, the following implications hold.\\
(i) Let   $(M,N),(M^{'},N^{'})\in g_{z}KD(T)$  such that $T_{M},$ $T_{M^{'}}$ are invertible and    $\sigma(T_{M})\cap\sigma(T_{N})=\sigma(T_{M^{'}})\cap\sigma(T_{N^{'}})=\emptyset.$  If $(T_{M})^{-1}\oplus 0_{N}=(T_{M^{'}})^{-1}\oplus 0_{N^{'}},$   then $(M,N)=(M^{'},N^{'}).$\\
(ii) Let   $\sigma,\sigma^{'}$ two spectral sets of $T$  such that  $0\notin \sigma(T)\setminus(\sigma\cap\sigma^{'})$    and  $(\sigma\cup\sigma^{'})\setminus\{0\}\subset \mbox{iso}\,\sigma(T).$ If $(T+rP_{\sigma})^{-1}(I-P_{\sigma})=(T+r^{'}P_{\sigma^{'}})^{-1}(I-P_{\sigma^{'}});$ where $P_{\sigma}$ is the spectral projection of $T$  associated to $\sigma,$ $|r|>\underset{\lambda\in  \sigma}{\mbox{max}}\,|\lambda|$ and $|r^{'}|>\underset{\lambda\in  \sigma^{'}}{\mbox{max}}\,|\lambda|,$    then $\sigma=\sigma^{'}.$
\end{prop}
\begin{proof} (i)   From the proof of Lemma \ref{lemomom}, we have $(M,N)=(\mathcal{N}(P_{\sigma}),\R(P_{\sigma}))$ and $(M^{'},N^{'})=(\mathcal{N}(P_{\sigma^{'}}),\R(P_{\sigma^{'}}));$ where $\sigma=\sigma(T_{N})$ and $\sigma^{'}=\sigma(T_{N^{'}}).$  As  $(T_{M})^{-1}\oplus 0_{N}=(T_{M^{'}})^{-1}\oplus 0_{N^{'}}$ then $\sigma(T_{M})=\sigma(T_{M^{'}})$ and thus  $\sigma(T_{N})=\sigma(T_{N^{'}}).$   This  proves  that $(M,N)=(M^{'},N^{'}).$\\
(ii) Follows  from (i).
\end{proof}
The previous Proposition \ref{prop.unicinverse} gives a sense to the next remark.
\begin{rema}\label{remadoublecomm}   If  $T\in L(X)$  is $g_{z}$-invertible, then\\
(i)  For every $(M,N)\in g_{z}KD(T)$ such that $T_{M}$ is invertible and  $\sigma(T_{M})\cap\sigma(T_{N})=\emptyset,$ the $g_{z}$-inverse  operator   $S:=(T_{M})^{-1}\oplus 0_{N} \in \mbox{comm}^{2}(T),$ and we call  $S$ the  $g_{z}$-inverse of $T$ associated to $(M,N).$\\
(ii) If  $\sigma$  is  a spectral set of $T$ such that $0\notin \sigma(T)\setminus\sigma$    and  $\sigma\setminus\{0\}\subset \mbox{iso}\,\sigma(T),$  then the  operator $T^{D}_{\sigma}:=(T+rP_{\sigma})^{-1}(I-P_{\sigma})\in \mbox{comm}^{2}(T)$ is a $g_{z}$-inverse of $T;$ where $|r|>\underset{\lambda\in  \sigma}{\mbox{max}}\,|\lambda|,$ and we call      $T^{D}_{\sigma}$ the $g_{z}$-inverse of $T$ associated to $\sigma.$
\end{rema}
If $T\in L(X)$ is generalized Drazin invertible which is not invertible, then by  \cite[Lemma 2.4]{koliha} and Proposition \ref{prop.unicinverse} we conclude that the  Drazin inverse $T^{D}$ (see \cite{koliha}) is exactly the $g_{z}$-inverse of $T$  associated to $\sigma=\{0\}.$

\medskip

\begin{prop} Let  $T, S \in L(X)$  two  commuting $g_{z}$-invertible. If  $\sigma$    and $\sigma^{'}$ are  spectral  sets of $T$ and $S,$ respectively such that  $0\notin(  \sigma(T)\setminus\sigma) \cup (\sigma(S)\setminus\sigma^{'}),$   $\sigma\setminus\{0\}\subset\mbox{iso}\,\sigma(T)$ and $\sigma^{'}\setminus\{0\}\subset\mbox{iso}\,\sigma(S),$  then        $T, S, T^{D}_{\sigma},S^{D}_{\sigma^{'}}$   are mutually commutative.
\end{prop}
\begin{proof}
As $TS=ST$ then the  previous remark entails that $T^{D}_{\sigma}=(T+rP_{\sigma})^{-1}(I-P_{\sigma}) \in \mbox{comm}(S^{D}_{\sigma^{'}}),$   and analogously for other operators.
\end{proof}
The following proposition  describe the relation between the $g_{z}$-inverse of a $g_{z}$-invertible operator $T$  associated to $(M,N)$ and the $g_{z}$-inverse of $T$ associated to a spectral set $\sigma.$ It's proof is clear.
\begin{prop} If  $T \in L(X)$   is  $g_{z}$-invertible, then the following are equivalent.
\begin{enumerate}[nolistsep]
\item[(i)] $S$ is a  $g_{z}$-inverse of $T$  associated to $(M,N);$
\item[(ii)]  $\sigma:=\sigma(T_{N})$ is a spectral set of $T,$ $S$ is  the    $g_{z}$-inverse of $T$ associated to  $\sigma$ and $M=\mathcal{N}(P_{\sigma}).$
\end{enumerate}
In other words, if  that  is the case then     $T^{D}_{\sigma(T_{N})}=(T_{M})^{-1}\oplus 0_{N}.$
\end{prop}
Our  next theorem gives a generalization of \cite[Theorem 4.4]{koliha} in the case of the complex  Banach algebra $L(X).$ Denote by $\mbox{Hol}(T)$ the set of all analytic functions defined on an open neighborhood of $\sigma(T).$
\begin{thm}\label{thminverseanalytic} If $0\in  \sigma(T)\setminus\mbox{acc}(\mbox{acc}\,\sigma(T)),$ then for every spectral set $\sigma$ such that $0\in  \sigma$ and $\sigma\setminus\{0\}\subset\mbox{iso}\,\sigma(T)$ we have
$$T^{D}_{\sigma}=f_{\sigma}(T),$$
where  $f_{\sigma}\in \mbox{Hol}(T)$    is  such that  $f_{\sigma}=0$  in a neighborhood of $\sigma$ and $f_{\sigma}(\lambda) = \frac{1}{\lambda}$ in a neighborhood of $\sigma(T)\setminus\sigma.$  Moreover $\sigma(T^{D}_{\sigma})=\{0\}\cup\{\lambda^{-1} : \lambda \in \sigma(T)\setminus \sigma\}.$
\end{thm}
\begin{proof} Let  $\Omega_{1}$ and $\Omega_{2}$   two  disjoint open sets such that $\sigma\subset \Omega_{1}$   and $\sigma(T)\setminus\sigma\subset \Omega_{2}$  (for the   construction of $\Omega_{1}$ and $\Omega_{2},$ see the  paragraph below) and let  $g \in \mbox{Hol}(T)$ be the function defined by
\[ g(\lambda) = \left\{ \begin{array}{ll}
         1 & \mbox{ if   $\lambda \in \Omega_{1}$}\\
         0 & \mbox{ if   $\lambda \in \Omega_{2}$}
\end{array} \right. \]
 It is clear that   $P_{\sigma}=g(T)$ and    as  $T^{D}_{\sigma}=(T+rP_{\sigma})^{-1}(I-P_{\sigma})$  (where $|r|>\underset{\lambda\in  \sigma}{\mbox{max}}\,|\lambda|$ be arbitrary), then the function  $f_{\sigma}(\lambda)=(\lambda +rg(\lambda))^{-1}(1-g(\lambda))$ has the required property. Moreover, we have     $\sigma(T^{D}_{\sigma})=f_{\sigma}(\sigma(T))=\{0\}\cup\{\lambda^{-1} : \lambda \in \sigma(T)\setminus \sigma\}.$
\end{proof}
According to   \cite{gohbergbook},   if   $\sigma$ is a spectral set of $T$  then there exist two  disjoint open sets $\Omega_{1}$ and $\Omega_{2}$   such that $\sigma\subset \Omega_{1}$ and $\sigma(T)\setminus\sigma\subset \Omega_{2}.$   Choose a Cauchy domains $S_{1}$ and $S_{2}$ such that $\sigma\subset S_{1},$ $\sigma(T)\setminus\sigma\subset S_{2},$  $\overline{S_{1}}\subset\Omega_{1}$ and $\overline{S_{2}}\subset\Omega_{2}.$ It follows that  the  spectral projection  corresponding to $\sigma$ is
$$P_{\sigma}=\frac{1}{2i\pi}\int_{\partial  S_{1}}(\lambda I- T)^{-1}d\lambda.$$  Moreover,   if $0\in  \sigma$  and $\sigma\setminus\{0\}\subset\mbox{iso}\,\sigma(T)$ then  from  Theorem  \ref{thminverseanalytic} we conclude that
$$T^{D}_{\sigma}=\frac{1}{2i\pi}\int_{\partial  S_{2}}\frac{1}{\lambda}(\lambda I- T)^{-1}d\lambda.$$

\section{Weak SVEP and applications}
Our   next theorem gives  new  characterization of some Browder's theorem  type classes in terms of spectra introduced and studied in this paper. This theorem is an improvement  of some recent results  dressed in \cite{gupta-kumar,karmouni-tajmouati}.
\begin{thm}\label{thm.j1} Let $T \in L(X).$ Then\\
(i)  $T \in (B)$ if and only if $\sigma_{g_{z}w}(T)=\sigma_{g_{z}d}(T).$\\
(ii)  $T \in (B_{e})$ if and only if $\sigma_{g_{z}f}(T)=\sigma_{g_{z}d}(T).$\\
(iii) $T \in (aB)$ if and only if $\sigma_{ug_{d}w}(T)=\sigma_{lg_{d}d}(T).$
\end{thm}
\begin{proof} (i) Let $\lambda \notin \sigma_{g_{z}w}(T).$ From Corollary \ref{coracc},  $\lambda  \notin \mbox{acc}\,\sigma_{pbw}(T)$ [note that $\mbox{acc}\,\sigma_{pbw}(T-\lambda I)=\mbox{acc}\,(\sigma_{pbw}(T))-\lambda$].   As $T \in (B),$ \cite[Theorem 2.6]{karmouni-tajmouati} (see also  \cite[Theorem 2.8]{gupta-kumar}) implies that $\lambda \notin \mbox{acc}\,\sigma_{gd}(T)$ and  from  Theorem  \ref{thm.i10},  it follows that  $\lambda \notin \sigma_{g_{z}d}(T).$    As the   inclusion  $\sigma_{g_{z}w}(T)\subset \sigma_{g_{z}d}(T)$  is always true then  $\sigma_{g_{z}w}(T)=\sigma_{g_{z}d}(T).$  Conversely,  let $\lambda \notin \sigma_{w}(T),$  then $\lambda \notin \sigma_{g_{z}w}(T)=\sigma_{g_{z}d}(T).$ On the other hand,  \cite[Corollary 3.7]{aznay-ouahab-zariouh2} implies that  there exists $(M,N) \in \mbox{Red}(T)$ such that   $\mbox{dim}N<\infty,$ $T_{M}-\lambda I$ is Weyl and   semi-regular with  $T_{N}-\lambda I$ is nilpotent.  Since  $T-\lambda I$ is  $g_{z}$-invertible then  $p(T_{M}-\lambda I)=\tilde{p}(T-\lambda I)=\tilde{q}(T-\lambda I)=q(T_{M}-\lambda I)=0,$ and so   $T_{M}-\lambda I$ is invertible.  Hence  $T-\lambda I$ is Browder and consequently  $T\in (B).$\\
By   \cite[Corollary 2.10]{karmouni-tajmouati} (see also \cite[Corollary 2.14]{gupta-kumar}), the point (ii)   go similarly with (i). Using  \cite[Theorem 2.7]{karmouni-tajmouati}, we obtain analogously the  point (iii).
\end{proof}
\begin{dfn}
Let $A$ be a  subset of $\mathbb{C}.$ We say that $T\in L(X)$ has  the Weak SVEP on $A$ ($T$ has the W$_{A}$-SVEP for brevity)  if there exists a subset $B \subset A$ such that $T$ has the SVEP on $B$ and $T^*$ has the SVEP on $A\backslash B$. If $T$  has  the  W$_{\mathbb{C}}$-SVEP then $T$ is said to have   the Weak SVEP ($T$ has the W-SVEP for brevity).
\end{dfn}
\begin{rema} (i) Let $A$ be a  subset of $\mathbb{C}.$ Then $T\in L(X)$ has  the W$_{A}$-SVEP if and only if for every $\lambda \in A,$ the operator $T$ or $T^{*}$ has the SVEP at $\lambda.$\\
(ii) If $T$ or $T^{*}$ has  the SVEP then $T$ has the W-SVEP. But the converse is not generally  true. The left shift operator $L\in L(\ell^{2}(\N))$ defined by $L(x_{1}, x_{2},\dots)=(x_{2},x_{3},\dots),$ has the W-SVEP but it  does not have the SVEP.\\
(iii) The operator $L\oplus L^{*}$ does not have the W-SVEP.
\end{rema}
The next theorem gives a sufficient condition for an operator    $T \in L(X)$ to   have   the W-SVEP.
\begin{thm}\label{thmwsvep} Let $T \in L(X).$  If   $$X_{T}(\emptyset)\times X_{T^{*}}(\emptyset)\bigsubset[1.3] \{(x,0): x\in X\}\bigcup \{(0,f): f\in X^{*}\},$$   then $T$ has  the  W-SVEP.
\end{thm}
\begin{proof} Let $\lambda \in \C$ and  $V,W \subset \C$   two open neighborhood of  $\lambda.$  Let  $f : V \longrightarrow  X$  and   $g : W \longrightarrow  X^{*}$  two analytic functions such that $(T - \mu I )f (\mu) = 0$ and $(T^{*}-v I )g(v) = 0,$ for every $(\mu, v) \in V\times W .$ If we take   $U:=V\cap W$   then  \cite[Theorem 2.9]{aiena} implies that $\sigma_{T}(f(\mu))=\sigma_{T}(0)=\emptyset=\sigma_{T^{*}}(0)=\sigma_{T^{*}}(g(\mu)),$  for every $\mu \in U.$    Hence $(f(\mu),g(v))\in X_{T}(\emptyset)\times X_{T^{*}}(\emptyset),$ for every $\mu, v \in U.$  So two cases are discussed. The first    there exists $\mu \in U$ such that  $g(\mu)\neq 0.$ As $(f(v),g(\mu))\in X_{T}(\emptyset)\times X_{T^{*}}(\emptyset),$ for every $v \in U$  then  by hypotheses $f\equiv 0$ in $U.$ The identity theorem for analytic functions  entails  that     $T$ has the SVEP at $\lambda.$ The second  $g(\mu)= 0,$ for every $\mu \in U.$     In  the same way, we prove that  $T^{*}$ has the SVEP at $\lambda.$  Hence $T$ has the W-SVEP.
\end{proof}
\noindent {\bf Question:} Similarly  to \cite[Theorem 2.14]{aiena} which characterizes  the SVEP of  $T \in L(X)$ in terms of its  local spectral subspace $X_{T}(\emptyset),$ we ask if the converse of  Theorem \ref{thmwsvep} is true?

\medskip
The next proposition    characterizes    the classes $(B)$ and $(aB)$ in terms  of the Weak SVEP.
\begin{prop}\label{propj.3}
  Let  $T \in L(X).$   The following hold.\\
(a)  For $\sigma_{*}\in \{\sigma_{w}, \sigma_{bw},  \sigma_{g_{z}w}\},$ the following statements are equivalent.
   \begin{enumerate}[nolistsep]
     \item[(i)] $T\in (B);$
     \item[(ii)]   $T$  has   the  Weak SVEP on  $\sigma_{*}(T)^{C};$
     \item[(iii)]  For all $\lambda \notin \sigma_{*}(T),$  $T\oplus T^{*}$ has  the  SVEP at $\lambda;$
     \item[(iv)]  For all $\lambda \notin \sigma_{*}(T),$  $T$  has  the  SVEP at $\lambda;$
     \item[(v)]  For all $\lambda \notin \sigma_{*}(T),$  $T^*$ has  the  SVEP at $\lambda.$
\end{enumerate}
(b)   For $\sigma_{*}\in \{\sigma_{e}, \sigma_{bf},  \sigma_{g_{z}f}\},$ the following statements are equivalent.
   \begin{enumerate}[nolistsep]
     \item[(i)]  $T\in (B_{e});$
     \item[(ii)]  For all $\lambda \notin \sigma_{*}(T),$  $T\oplus T^*$ has  the  SVEP at $\lambda.$
   \end{enumerate}
(c) For $\sigma_{*}\in \{\sigma_{uw}, \sigma_{ubw},   \sigma_{ug_{z}w}\},$ the following statements are equivalent.
   \begin{enumerate}[nolistsep]
     \item[(i)] $T\in (aB);$
     \item[(ii)]  $T$  has   the  Weak SVEP on  $\sigma_{*}(T)^{C};$
     \item[(iii)]  For all $\lambda \notin \sigma_{*}(T),$  $T\oplus T^{*}$  has  the  SVEP at $\lambda;$
     \item[(iv)]  For all $\lambda \notin \sigma_{*}(T),$  $T$  has  the  SVEP at $\lambda;$
     \item[(v)]  For all $\lambda \notin \sigma_{*}(T),$  $T^*$ has  the  SVEP at $\lambda.$
\end{enumerate}
\end{prop}
   \begin{proof}
(a)  For $\sigma_{*}=\sigma_{g_{z}w},$ we have only to show the implication  ``(ii)  $\Longrightarrow$  (i)".  The several  other implications  are clair. Let $\lambda \notin \sigma_{g_{z}w}(T)$ then there  exist $(M,N)\in \mbox{Red}(T)$ such that  $(T-\lambda I)_{M}$ is Weyl   and $(T-\lambda I)_{N}$ is zeroloid. Hence  $T$ or $T^*$ has  the  SVEP at $\lambda$ is equivalent to say that  $T_{M}$ or $(T_{M})^{*}$ has  the  SVEP at $\lambda$ which is equivalent to say that   $\mbox{min}\,\{p((T-\lambda I)_{M}),q((T-\lambda I)_{M})\}<\infty.$ Therefore  $(T-\lambda I)_{M}$  is Browder and then $\lambda \notin \sigma_{g_{z}d}(T).$ Thus  $T \in (B).$\\
 For $\sigma_{*}\in \{\sigma_{w}, \sigma_{bw}\},$ the  implication  ``(ii)  $\Longrightarrow$  (i)"  is analogous to the one of   $\sigma_{*}=\sigma_{g_{z}w}.$   The other implications are already done  in  \cite{aiena}.  The assertions  (b) and (c) (in which some implications are already done in \cite{aiena,aznay-ouahab-zariouh3,gupta-kumar,karmouni-tajmouati}) go similarly with  (a).
   \end{proof}
The next theorem extends  \cite[Theorem 5.6]{aiena}.
\begin{thm}   Let  $T \in L(X)$ such that $\mbox{int}\,\sigma_{g_{z}w}(T)=\emptyset.$  The following statements are equivalent.
 \begin{enumerate}[nolistsep]
\item[(i)] $T \in (B);$
\item[(ii)] $T \in (B_{e});$
\item[(iii)] $T \in (aB);$
\item[(iv)] $T$ has  the SVEP;
\item[(v)] $T^{*}$ has the SVEP;
\item[(vi)] $T\oplus T^{*}$ has  the SVEP;
\item[(vii)]  $T$ has  the W-SVEP.
\end{enumerate}
\end{thm}
\begin{proof} (i) $\Longrightarrow$ (vi)  As $T \in (B),$ Proposition \ref{propj.3} entails that     $T\oplus T^{*}$ has the SVEP on  $\sigma_{g_{z}w}(T)^{C}.$  Let   $\lambda \in \sigma_{g_{z}w}(T),$     $U \subset \C$ be  an  open neighborhood of  $\lambda$ and      $f : U \longrightarrow  X$  be  an analytic function which satisfies   $(\mu I  - T )f (\mu) = 0,$ for every $\mu \in U.$ The hypothesis that   $\mbox{int}\,\sigma_{g_{z}w}(T)=\emptyset$ implies that   there exists $\gamma \in U\cap  (\sigma_{g_{z}w}(T))^C.$ Let $V \subset U$  be  an neighborhood open  of $\gamma$  then $f \equiv 0$ on  $V,$   since $T$  has  the  SVEP at $\gamma.$  It then follows from   the identity theorem for analytic functions that   $f \equiv 0$ on  $U.$
So $T$ has the SVEP at $\lambda,$ and   analogously we prove that $T^{*}$ has the SVEP at $\lambda.$  Consequently  $T\oplus T^{*}$ has  the SVEP.  It is clear  that  the statement $(vi)$ implies without condition on $T$ all  other statements and   under the assumption $\mbox{int}\,\sigma_{g_{z}w}(T)=\emptyset,$  all statements imply the statement $(i).$   This completes the proof.
\end{proof}
\goodbreak

{\small \noindent Zakariae Aznay,\\  Laboratory (L.A.N.O), Department of Mathematics,\\Faculty of Science, Mohammed I University,\\  Oujda 60000 Morocco.\\
aznay.zakariae@ump.ac.ma\\

{\small \noindent Abdelmalek Ouahab,\\  Laboratory (L.A.N.O), Department of Mathematics,\\Faculty of Science, Mohammed I University,\\  Oujda 60000 Morocco.\\
ouahab05@yahoo.fr\\

 \noindent Hassan  Zariouh,\newline Department of
Mathematics (CRMEFO),\newline
 \noindent and laboratory (L.A.N.O), Faculty of Science,\newline
  Mohammed I University, Oujda 60000 Morocco.\\
 \noindent h.zariouh@yahoo.fr

\end{document}